\documentclass[reqno]{amsart}
\usepackage[english]{babel}
\usepackage[utf8]{inputenc}

\usepackage{xcolor, mathrsfs}
\usepackage{enumerate, amssymb}

\definecolor{darkgreen}{rgb}{0,0.6,0}
\definecolor{darkred}{rgb}{0.7,0,0}
\definecolor{darkblue}{rgb}{0,.2,.7}
 \usepackage[%
   colorlinks, 
  citecolor=darkgreen, linkcolor=darkblue,
   pdfpagelabels=true,
   unicode=true,
   pdfusetitle
 ]{hyperref}

\hypersetup{
 pdfauthor={Nadine Große and Simone Murro},
 pdftitle={The well-posedness of the Cauchy problem for the Dirac operator on globally hyperbolic manifolds with timelike boundary},
 pdflang={en}
}

\usepackage{relsize}

\usepackage{tikz}
\usetikzlibrary{snakes}
\usetikzlibrary{patterns}

\newtheorem{thm}{Theorem}[section]
\newtheorem{corollary}[thm]{Corollary}
\newtheorem{lemma}[thm]{Lemma}
\newtheorem{proposition}[thm]{Proposition}

\theoremstyle{remark}
\theoremstyle{definition}
\newtheorem{remark}[thm]{Remark}
\newtheorem{definition}[thm]{Definition}
\newtheorem{example}[thm]{Example}

\numberwithin{equation}{section}
\allowdisplaybreaks[1]

\catcode`@=12

\renewcommand\thanks[1]{%
  \begingroup
  \renewcommand\thefootnote{}\footnote{#1}%
  \addtocounter{footnote}{-1}%
  \endgroup
}

\renewcommand{\tilde}{\widetilde}
\renewcommand{\epsilon}{{\varepsilon}}

\newcommand{\Dir}{\mathsf{D}}
\newcommand{\Ham}{\mathsf{H}}
\newcommand{\Green}{\mathsf{G}}
\newcommand{\Mbc}{\mathsf{M}}
\newcommand{\oM}{\Mbc}
\newcommand{\oK}{\mathlarger{\mathfrak{S}}}
\newcommand{\oE}{{\mathfrak{E}}}
\newcommand{\oH}{{\mathfrak{G}}}
\newcommand{\oL}{{\mathfrak{L}}}

\newcommand{\Id}{{\textnormal{Id}}}

\newcommand{\f}{{\mathfrak{f}}}
\newcommand{\h}{{\mathfrak{h}}}

\newcommand{\RR}{{\mathbb{R}}}
\newcommand{\CC}{{\mathbb{C}}}
\let\<\langle
\let\>\rangle

\newcommand{\R}{{\mathcal{R}^\wedge}}

\newcommand{\Rv}{{\mathcal{R}^\vee}}
\newcommand{\V}{{\mathcal{V}}}

\newcommand{\bR}{{\partial\mathcal{R}^\wedge}}

\newcommand{\bM}{{\partial\mathcal{M}}}

\newcommand{\M}{{\mathcal{M}}}
\newcommand{\N}{{\mathcal{N}}}

\newcommand{\T}{{\mathcal{T}}}

\newcommand{\n}{{\mathtt{n}}}
\newcommand{\cH}{{\mathscr{H}}}
\renewcommand{\i}{\imath}
\newcommand{\supp}{{\textnormal{supp\ }}}

\newcommand{\define}{\mathrel{\rm:=}}

\newcommand{\fiber}[2]{\langle  #1\,|\, #2  \rangle}
\newcommand{\scalar}[2]{(#1\, |\, #2)}

\author[N. Gro{\ss}e]{Nadine Gro{\ss}e} \address{N. Gro{\ss}e, Mathematisches Institut,
 Universit\"at Freiburg, 79104 Freiburg, Germany}
\email{nadine.grosse@math.uni-freiburg.de}

\author[S. Murro]{Simone Murro} \address{S. Murro, Mathematisches Institut,
 Universit\"at Freiburg, 79104 Freiburg, Germany}
\email{simone.murro@math.uni-freiburg.de}

\begin{document}

\title[The well-posedness of the Cauchy problem for the Dirac operator]{The well-posedness of the Cauchy problem for the Dirac operator on globally hyperbolic manifolds with timelike boundary}

 \thanks{Both authors are supported by the research grant { ``Geometric boundary value problems for the Dirac operator''} of the Juniorprofessurenprogramm Baden-W\"urttemberg. S.M. is partially supported within the DFG research training group GRK 1821 ``Cohomological Methods in
Geometry''.}


\begin{abstract}
We consider the Dirac operator on globally hyperbolic manifolds
with timelike boundary and show well-posedness of the Cauchy initial boundary value problem coupled to MIT-boundary conditions. This is achieved by transforming the problem locally into a symmetric positive hyperbolic system, proving existence and uniqueness of weak solutions and then  using local methods developed by Lax, Phillips and Rauch, Massey to show smoothness of the solutions.
Our proof actually works for a slightly more general class of local boundary conditions.
\end{abstract}
\maketitle

\section{Introduction}

The well-posedness of the Cauchy problem for the Dirac operator on a Lorentzian manifold $(\overline \M, g)$ is a classical problem which has been exhaustively studied in many contexts. If the underlying background is \textit{globally hyperbolic}, a complete answer is known: In \cite{dimock} it was shown that a fundamental solution for the Dirac equation can obtained from a fundamental solution of the conformal wave operator via the Lichnerowicz formula. Since the Cauchy problem for the conformal wave operator is well-posed \cite{Choquet,Leray}, it follows that the Cauchy problem for the Dirac operator is also well-posed. See also \cite{bgp} for a direct treatment of Cauchy problems for Dirac operators.\medskip

Even if there exists a plethora of models in physics where globally hyperbolic spacetimes have been used as a background, there also exist many applications which require a manifold with non-empty boundary. Indeed, recent developments in quantum field theory focused their attention on manifolds with timelike boundary \cite{aqft,Zahn}, e.g.  anti-de Sitter spacetime \cite{Dappia1,Dappia2} and BTZ spacetime \cite{Dappia3}. Moreover, experimental setups for
studying the Casimir effect enclose (quantum) fields between walls, which may be mathematically described by introducing timelike boundaries  \cite{DappiaCas}. Also moving walls in a spatial set-up correspond to a timelike boundary in the Lorentzian manifold. In these settings, the correspondence between the well-posedness of the conformal wave operator and the Dirac operator breaks down due to the boundary condition. Even if the Cauchy problem for the conformal wave operator is proved to be well-posed for a large class of boundary conditions for stationary spacetimes  \cite{Dappia},  it is still not clear how to relate this with the Cauchy problem for the Dirac operator. In fact, solvability in the Dirac case very much depends on the boundary condition, e.g. if the Dirichlet boundary condition is applied to spinors on the boundary, in general there does not exist any smooth solution to the Dirac equation.\medskip

The goal of this paper is to investigate the well-posedness of the Cauchy problem for the Dirac operator in globally hyperbolic manifolds with timelike boundary  $\M$ in the sense of Definition~\ref{def: globally hyperbolic with timelike boundary}. Coming from the 'moving wall'-picture our time function and splitting should by induced from an exterior globally hyperbolic manifold: Let $(\overline{\M},g)$ be a globally hyperbolic spin manifold of dimension $n+1$.  Let $\N$ be an $n$-dimensional Lorentzian submanifold of $\overline{\M}$.  We assume additionally that $\N$ divides $\overline{\M}$ into two (not empty) connected components.  In Section~\ref{sec:geom} we see that each of these connected components is then globally hyperbolic manifolds with timelike boundary and that every globally hyperbolic manifolds with timelike boundary arises that way. In particular, there we see that any Cauchy time function  $t\colon \overline{\M}\to \RR$ on $\overline{\M}$ induces one on the connected component. Thus we can work with the splitting and time function induced from $\overline{\M}$ in the following: Then $\overline{\Sigma}_s\define t^{-1}(s)$ is a smooth spacelike Cauchy surface of  $\overline{\M}$ and,  see Section~\ref{sec:geom}, $\widehat{\Sigma}\define \overline \Sigma\cap \N$ is a Cauchy surface for $\N$. In particular,  $\{t^{-1}(s)\}_{s\in \RR}$ gives a foliation by Cauchy surfaces, and we set $\Sigma_s\define t^{-1}(s)\cap \mathcal{M}$. \medskip 

We always assume that we fix the spin structure on $\M$.  Let $S\M$ denote the spinor bundle over $\M$ and $S\Sigma_0$ the induced spinor bundle over $\Sigma_0$. Moreover, let $\Dir$ be the Dirac operator on $S\M$, for details and notation see Section~\ref{sec:preliminaries}.\medskip
 
 Our main result is a well-posedness theorem for the Dirac operator with MIT-boundary conditions. The \textit{MIT boundary condition} is a local boundary condition that was introduced for the first time in \cite{MIT1} in order to reproduce the confinement of quark in a finite region of space: ``Dirac waves'' are indeed reflected on the boundary. Few years later, it was used in the description of hadronic states like baryons \cite{MIT2} and mesons \cite{MIT3}. More recently, the MIT boundary condition were employed in \cite{Casimir} for computing the Casimir energy in a a three-dimensional rectangular box, in \cite{FKSY} and \cite{FR1,FR2} in order to construct an integral representation for the Dirac propagator in Kerr-Newman and Kerr Black Hole Geometry respectively, and in \cite{IR} for proving the asymptotic completeness for linear massive Dirac fields on the Schwarzschild Anti-de Sitter spacetime.\medskip
 
 In this paper we obtain
 
  \begin{thm}\label{maintheorem}
The Cauchy problem for the Dirac operator with MIT-boundary condition on a globally hyperbolic spin manifold $\M$ with timelike boundary $\bM$ is well-posed, i.e., for any $f\in\Gamma_{cc}(S\M)$ and $h\in\Gamma_{cc}(S\Sigma_0)$   there exists a unique smooth solution $\psi$  with spatially compact support to the mixed initial-boundary value problem
\begin{equation}\label{CauchyDir}
\begin{cases}
\Dir \psi=f   \\
\psi|_{\Sigma_0} ={h} \\
(\gamma({\n})  - \imath ) \psi|_{\bM} =0
\end{cases}
\end{equation}
 which depends continuously on the data $(f,h)$. Here, $\gamma(\n)$ denotes Clifford multiplication with ${\n}$ the outward unit normal on $\bM$ and $\Gamma_{cc}(.)$ denotes the space of sections that are compactly supported in the interior of the underlying manifold.
\end{thm}

For a subclass of stationary  spacetimes  with timelike boundary admitting a suitable timelike Killing vector field Theorem~\ref{maintheorem} was already proven in \cite{FR1}.

\begin{remark}\label{rem:bd}
Theorem~\ref{maintheorem} holds under more general local boundary conditions, namely replacing $\Mbc=\gamma(\n)-\imath$ by any linear non-invertible map $\Mbc\colon \Gamma(S\bM) \to \Gamma(S\bM)$ with constant kernel dimension and such that $\Mbc\psi|_\bM=0$ and $\Mbc^\dagger\psi|_\bM=0$, see Remark~\ref{rem:ad} for the definition of the adjoint boundary condition  $\Mbc^\dagger$,  both imply
\begin{equation}\label{general bound cond}
  \fiber{\psi}{\gamma(e_0)\gamma(\n)\psi}_q =0
\end{equation}
 for all $q\in \bM$, see also Remark~\ref{rem:bd2}. Here $e_0$ is the globally defined unit timelike future pointing vector field defined by~\eqref{def:e_0}.
\end{remark}

As usual,  the well-posedness of the Cauchy problem will  guarantee the existence of \textit{Green operators} for the Dirac operator. In globally hyperbolic spin manifolds with empty boundary, these operators play a fundamental role in the quantization of linear field theory~\cite{DHP,dimock}. In loc. cit., the quantization of
a free field theory is interpreted as a two-step procedure: The first consists of the assignment to a physical system of a  $^*$-algebra of observables which
encodes structural properties such as causality, dynamics and the canonical anticommutation relations.
The second step calls for the identification of an algebraic state, which is a positive, linear
and normalized functional on the algebra of observables.
This quantization scheme goes under the name of \emph{algebraic quantum field theory} and it is especially well-suited for formulating quantum theories for Green-hyperbolic operators also on manifold -- see e.g.~\cite{aqft2,gerard} for textbook, to~\cite{BG,BD,FK} for recent reviews,  and~\cite{simo5,simo4,DHP,simo6,simo3,simo1,simo2} for some applications.

\begin{remark}The Cauchy problem~\eqref{CauchyDir} is still well-posed in a larger class of initial data. But then some compatibility condition for $f$ and $h$ on $\partial\Sigma_0$ is needed---see Remark~\ref{star}.
In the subclass of stationary spacetimes considered in \cite{FR1}, as mentioned above, this compatibility condition reduces to the one therein. Without these compatibility conditions the solution would still exist but the singularities contained in a neighborhood of $\partial\Sigma_0$ would propagate with time along lightlike geodesics. 
In case, they hit again the boundary some boundary phenomena as reflection will occur.  For future work it is of course interesting to obtain  an explicit method  to  construct the  corresponding Green operators and to obtain more information on how singularities behave when hitting the boundary.
\end{remark}

Our strategy to prove the well-posedness of the Cauchy problem is as follows: Firstly
we restrict to a  strip of finite time and prove general properties of smooth solutions as finite propagation of speed and uniqueness in Section~\ref{sec:fps}. Using the theory of symmetric positive hyperbolic systems, see e.g. \cite{Fr2}, we investigate properties of a weak solution in Section~\ref{sec:ws}. 
In Section~\ref{diffsol} we localize the problem by introducing suitable coordinates small enough that we can associate to our Dirac problem a hyperbolic system that fits into the class considered in \cite{LP, RM} and thus provides smoothness of a weak solution. Then we can use these properties to prove existence of a weak solution in an arbitrary  time strip.
With the help of uniqueness we can then easily glue together solutions on an arbitrary time strip to  obtain global ones in Section~\ref{sec:gs}.\medskip

\noindent \textit{Acknowledgements.} 
We would like to thank Claudio Dappiaggi, Nicol\'o Drago, Felix Finster, Umberto Lupo, Oliver Petersen and Christian R\"oken for helpful discussions. 

\section{Preliminaries}\label{preliminaries}

For a Lorentzian manifold $\M$, we use the convention that the metric $g$ of $\M$ has signature$ (-,+, \cdots,  +)$.

Globally hyperbolic as in \cite{bgp,Ge} means that the manifold admits a Cauchy surface, i.e.  
a hypersurface that is hit exactly once by every inextendable timelike curve. Moreover, this surface can always be chosen to be smooth and spacelike, see \cite{BS}. Without further mentioning it, we always assume that our manifolds are orientable and time-orientable. Note that for  $n+1=4$, the case relevant for physics, such globally hyperbolic manifolds are automatically spin.

\subsection{Globally Hyperbolic Manifolds with Timelike Boundary}\label{sec:geom}
The definition of what should be a globally hyperbolic manifold with timelike boundary was first proposed by Solis in~\cite{Sol}, see also~\cite{CGS}.
\begin{definition}\label{def: globally hyperbolic with timelike boundary}
A \emph{globally hyperbolic manifold with timelike boundary} is an $(n + 1)$-dimensional, oriented,
time-oriented Lorentzian manifold $(\M,g)$ with boundary $\bM$ such that
\begin{itemize}
\item[(i)] the pullback of $g$ respect the natural inclusion $\iota\colon \bM \to \M$ defines a Lorentzian metric $\iota^*g$ on the boundary;
\item[(ii)] $\M$ is causal, i.e. there are no closed causal curves;
\item[(iii)] for every point $p,q\in\M$, $J^+(p)\cap J^-(q)$ is compact, where $J^+(p)$ (\textit{resp.} $J^-(p)$) denotes the causal future (\textit{resp}. past) of $p\in M$.
\end{itemize}
\end{definition}

In \cite{AFS} Ak{\'e}, Flores and Sanchez proved that for a globally hyperbolic manifold with timelike boundary there  always exists a Cauchy time function (i.e., a continuous function $\M\to \mathbb R$ which increases strictly on any future-directed causal curve and  whose level sets are Cauchy surfaces, see \cite[p. 65]{Beem})  whose gradient  is tangent to the boundary and obtain the analogue of the usual spacetime splitting of globally hyperbolic manifolds.\medskip

Thinking of the occurrence of manifolds with timelike boundary as described  in the introduction, e.g. 'moving mirrors', it is more natural for us, to obtain our manifold as a part of a globally hyperbolic manifold without boundary and to work with the induced time function. We prove in the following that under natural assumptions this always leads to a globally hyperbolic manifold with timelike boundary as defined above and all such manifolds are obtained that way.\medskip 

In particular, we will see that  we obtain  \cite[Thm. 1.1]{AFS} except for the condition that the gradient of the Cauchy time function is tangent to the boundary. Most of this can of course be found along the lines of the proofs in  \cite{CGS,AFS, ake}. Especially that every globally hyperbolic manifold with timelike boundary arises that way is shown in proven in \cite[Cor. 5.8]{AFS}. But for further use we  decided to still formulate the following theorems here.

\begin{thm}\label{thm_char_globhyp}
 Let $(\tilde{\M}, \tilde{g})$ be a globally hyperbolic manifold (without boundary). Let $\iota\colon \N \hookrightarrow \tilde{\M}$ be an embedding of a timelike hypersurface such that  $\tilde{\M}\setminus \N$ has two connected components $\M_i$, $i=1,2$. Then, $\overline{\M}_i= \M_i\cup \N\subset \tilde{M}$ is a globally hyperbolic manifold with  timelike boundary and $(\N, \tilde{g}|_{\N})$ is globally hyperbolic. 
 Moreover, each  globally hyperbolic manifold with timelike boundary arises as a $\overline{\M}_i$ in a construction as above.
\end{thm}

\begin{proof} Property (i) in Definition~\ref{def: globally hyperbolic with timelike boundary} is satisfied since $\overline{\partial\M}_i:=\N$ is timelike by assumption.
Moreover, every closed causal curve in $\N$ or $\bar{\M}_i$ is one in $\M$ and, hence, does not exist. In order to show that property (iii) is satisfied, we first show that $\N$ and $\overline{\M}_i$ are strongly causal. Notice that, since
 $(\tilde{M}, \tilde{g})$ is globally hyperbolic,  it is in particular strongly causal \cite[p. 59]{Beem}. We first show, that $(\N, \tilde{g}|_{\N})$ is strongly causal as well: Let $p\in \N$ and let $U$ be an open neighborhood of $p$ in $\N$. Choose an open neighborhood  $\tilde{U}$ of $p$ in $\tilde{\M}$ such that $\tilde{U}\cap \N=U$. Since $\tilde{\M}$ is strongly causal, there is a neighborhood $\tilde{V}\subset \tilde{U}$ of $p$ in $\tilde{\M}$ such that any causal curve starting and ending in $\tilde{V}$ is completely contained in $\tilde{U}$. Let now $V\define \tilde{V}\cap \N$. Then any causal curve in $\N$ starting and ending in $V$ is in particular a causal curve of $\tilde{\M}$ starting and ending in $\tilde{V}$ and hence completely stays in $U=\tilde{U}\cap \N$.
Analogously, we see that $(\overline{\M}_i, \tilde{g})$ is strongly causal.\medskip 

Furthermore, we note that both $\N$ and $\overline{\M}_i$  inherit the  time-orientation 
from $\tilde{\M}$ and that $\mathcal{N}$ has to be closed since $\mathcal{M}\setminus \mathcal{N}$ has two connected components. \medskip 

Let $p,q\in \N$. Let $J_\pm^\N(.)$ denote that causal past/future inside $\N$ and $J_\pm^{\tilde{\M}}(.)$ the ones inside $\tilde{\M}$. Note that $J_\pm^\N(p)\subset J_\pm^{\tilde{\M}}(p)\cap \N$.  Hence $J_-^\N(p)\cap J_+^\N(q)$ is a subset of the compact set $J_-^{\tilde{\M}}(p)\cap J_+^{\tilde{\M}}(q)\cap \N$.  This implies that the closure of $J_-^\N(p)\cap J_+^\N(q)$ is compact in $\N$. Together with the strong causality of $\N$ \cite[Lemma 4.29]{Beem} implies that $J_-^\N(p)\cap J_+^\N(q)$ is already compact and, hence, $(\N, \tilde{g}|_{\N})$ is globally hyperbolic.

We note that \cite[Lemma 4.29]{Beem} has exactly the same proof if we allow the manifold to have boundary. Hence, the analogue arguments from above imply that $(\overline{\M}_i, \tilde{g})$ is globally hyperbolic as well.\medskip

Let now $(\mathcal{M}, g)$ be a globally hyperbolic manifold with timelike boundary. As stated before this is  \cite[Cor. 5.8]{AFS}. To provide a  short-cut let us shortly present the strategy of the proof: Every Lorentzian manifold $(\mathcal{M}, g)$ with boundary admits an extension to a Lorentzian manifold $(\hat{\mathcal{M}}\define  \mathcal{M} \cup_{\partial \mathcal{M}} (\partial  \mathcal{M}\times [0,\epsilon)) , \hat{g})$ such that $\hat{g}|_{\mathcal{M}}=g$  \cite[Theorem A.1]{Sol}. On a given manifold with boundary the set of metrics that make this manifold into a globally hyperbolic one with timelike boundary is open in the set of Lorentzian metrics in the  'time cone topology' (defined by the $<$-relation on the space of Lorentzian metric as in \cite[Section 6]{Ge}, see also \cite[below Def. 2.3]{AFS}) and hence in the fine $C^0$-topology: This can be seen exactly along the lines of the proof of Geroch in \cite{Ge}. Hence, there is a smooth function $\delta\colon \partial \mathcal{M}\to (0,\epsilon)$ such that $(\mathcal{M}_1\define \mathcal{M}\cup_{\partial \mathcal{M}}  \{ (x,t)\ |\ x\in \partial \mathcal{M}, 0< t\leq \delta(x)\}, \hat{g})$ is a globally hyperbolic manifold with timelike boundary. Now we can deform $\hat{g}$ on a small enough neighborhood of $\partial \M_1$ to have product structure on this subset of $\M_1\setminus \M$ and such that the new metric $g_1$ is still globally hyperbolic on $\M_1$. The double $(\hat{\M}_1, \hat{g}_1)$ of $(\M_1, g_1)$  is then a smooth Lorentzian manifold without boundary with $(\M,g)\subset (\hat{\M}_1, \hat{g}_1)$. The double $(\hat{\M}_1, \hat{g}_1)$ is itself globally hyperbolic since every causal curve between two points in $\partial \M_1\subset \hat{\M}_1$ can be mapped using the reflection symmetry of the double to a causal curve staying in $\M_1$.
\end{proof}

\begin{thm}
  Let $(\tilde{\M}, \tilde{g})$ be a globally hyperbolic manifold (without boundary). Let $\iota\colon \N \hookrightarrow \tilde{\M}$ be the embedding of a timelike hypersurface such that  $\tilde{\M}\setminus \N$ has two connected components $\M_i$, $i=1,2$. Let $t\colon \tilde{M}\to \mathbb R$ be a Cauchy time function for $\tilde{M}$. Then,  $t^{-1}(s)\cap \N$ is a Cauchy surface of $\N$  and $t^{-1}(s)\cap \overline{\mathcal{M}}_i$ one of $\overline{\mathcal{M}}_i$ for all $s\in \mathbb R$. Moreover, $t|_{\N}\colon \N\to \mathbb R$ is a Cauchy time time function of $\N$.    
\end{thm}

\begin{proof}From Theorem~\ref{thm_char_globhyp} we know that $\mathcal{N}$ with its induces metric is globally hyperbolic.

   For any $s$, the spacelike $\tilde{\Sigma}_s\define t^{-1}(s)$ and the timelike $\N$ automatically intersect transversally. Hence, if $\Sigma_s\define \tilde{\Sigma}_s\cap \N$ is nonempty, it has to be a submanifold of $\tilde{\M}$ of codimension 2 and a spacelike hypersurface of $\N$.  In particular, such $\Sigma_s$ is a Cauchy surface for $\N$: Let $\gamma$ be a future directed inextendable causal curve in $\N$.  Since $\N$ is closed, $\gamma$ has to be defined on an open interval, say $\gamma\colon (a,b) \to \N$.  Then, $\gamma$ is also a  future directed causal curve in $\tilde{\M}$. Assume that $\gamma$ is extensible as a future directed causal  curve in $\tilde{\M}$, w.l.o.g. let $\tilde{\gamma}\colon (a,b]\to  \tilde{\M}$. Since $\N$ is closed, $\tilde{\gamma}(b)\in \N$ which gives the  contradiction.   Hence, $J_-^\N (\Sigma_s)= t^{-1}(-\infty, s] \cap \N$ if $\Sigma_s$ is nonempty. At least for some $u\in \mathbb R$ there has to be a point $p\in \Sigma_u$ since $\N$ is nonempty. The arguments from above show that any  future directed inextendable causal curve $\gamma$ in $\N$ through $p$ is already a future directed inextendable causal curve in $\tilde{\M}$ and, hence, intersects all $\tilde{\Sigma}_s$.  Thus, $\Sigma_s$ is a Cauchy surface of $\N$ for all $s\in \mathbb R$ and $t|_{\N}\colon \N\to \mathbb R$ is a Cauchy time time function of $\N$.    
   Analogously one sees that $t^{-1}(s)\cap \overline{\M}_i$ is a Cauchy surface of $\overline{\M}_i$.
   \end{proof}

 \subsection{Spinorial preliminaries}\label{sec:preliminaries}

In the following we always assume that the spin structure is fixed.  We denote by $S\M$ the associated \textit{spinor bundle}, that is in particular a complex vector bundle with $N\define2^{\lfloor \frac{n+1}{2}\rfloor}$-dimensional fibers, denoted by $S_p\M$ for $p\in \M$, fiberwise endowed  with a hermitian product
$$\fiber{\cdot}{\cdot}\colon S_p\M \times S_p\M \to \CC \,$$
  and with a \textit{Clifford multiplication} $\gamma\colon T\M\to \text{End}(S\M)$ that is in particular fiber-preserving and satisfies for all $p\in \M$ and $u,v\in T_p\M$ that 
  \[ \gamma(u)\gamma(v)+\gamma(v)\gamma(u)=-2g(u,v) \text{Id}_{S_p\M}.\] 
We denote the by $\Gamma_c(\cdot)$, $\Gamma_{cc}(\cdot)$, $\Gamma_{sc}(\cdot)$ resp. $\Gamma(\cdot)$ the spaces of compactly supported, compactly supported in the interior, spacelike compactly supported resp. smooth sections of a vector bundle. 
The \textit{(classical) Dirac operator} $\Dir \colon \Gamma(S\M) \to \Gamma(S\M)$ is defined as the composition of the connection $\nabla$ on $S\M$, obtained as a lift of the Levi-Civita connection on $T\M$, and the Clifford multiplication times $\i$. Thus, in local coordinates this reads as
$$\Dir = \sum_{\mu=0}^{n} \imath \varepsilon_\mu \gamma(e_\mu) \nabla_{e_\mu}  $$
where  $(e_\mu)_{\mu=0,\dots,n}$ is a local Lorentzian-orthonormal frame of $T\M$, $\varepsilon_\mu=g(e_\mu,e_\mu)=\pm 1$.\medskip

We recall that after the choice of the Cauchy time function $t$ the metric on the  globally hyperbolic manifold $(\tilde{\M},g)$ (without boundary) can be written as $g=g|_{\Sigma_t}-\beta^2 dt^2$ where $\beta\colon \tilde{\M}\to \mathbb R$ is a positive smooth function \cite{BS}. This representation we will also use for $\bar{M}_i\subset \tilde{M}$ as in Theorem~\ref{thm_char_globhyp}. Thus, although the local orthonormal frame $e_\mu$ in general only exists locally, we  have a globally defined unit timelike vector field  
\begin{equation}\label{def:e_0}
e_0\define \frac{1}{\beta} \partial_t
\end{equation}   that we will use in the following. In particular we have $\gamma(e_0)^{-1}=\gamma(e_0)$.

\subsection{Reformulation as a symmetric positive hyperbolic system}
In this section we will formulate the Dirac equation~\eqref{CauchyDir} locally as a symmetric positive hyperbolic system. For that we shortly recall the basic definition, for more details see \cite{Fr1, Fr2}.\medskip

For the following definition, let $E\to\M$ be a real (or complex) vector bundle with finite rank $N$ endowed with  the canonical fiberwise metric $\fiber{\cdot}{\cdot}\define \fiber{\cdot}{\cdot}_{\mathbb R^N}$ (or $\fiber{\cdot}{\cdot}:=\fiber{\cdot}{\cdot}_{\mathbb C^N}$). Moreover, let us endow ${\Gamma_{cc}(E)}$ with the $L^2$-scalar product $$\scalar{\cdot}{\cdot}_{\M}\define\int_\M \fiber{\cdot}{\cdot} \text{Vol}_\M \,,$$
where $\text{Vol}_\M$ denotes the volume element. Moreover, let $\Vert.\Vert_{L^2(\M)}^2\define \scalar{.}{.}_{\M}$.

\begin{definition}\label{def:symm syst}
 A linear differential operator $\oL \colon \Gamma(E) \to \Gamma(E)$ of first order is called a \emph{symmetric system} over $\M$ if 
\begin{enumerate}
\item[(S)] the principal symbol $\sigma_\oL (\xi) \colon E_p \to E_p$ is hermitian with respect to $\fiber{\cdot}{\cdot}$ for every $\xi\in T^*_p\M$ and for every $p \in \M$.
\end{enumerate} 
  Additionally, we say that $\oL$ is \emph{positive} respectively \emph{hyperbolic} if it holds:
\begin{enumerate}
\item[(P)] The bilinear form $\fiber{ \, (\oL+\oL^\dagger ) \cdot}{\cdot}$ on $E_p$ is positive definite, where $\oL^\dagger$ denotes the formal adjoint of $\oL$ with respect to the $L^2$-product on $\Gamma_{cc}(E)$ and $\Re$ denotes the real part.
\item[(H)] For every future-directed timelike covector $\tau \in T_p^*\M$, the bilinear form $\fiber{\sigma_\oL (\tau) \cdot}{\cdot}$ is positive definite on $E_p$.
\end{enumerate}
\end{definition}
Let us recall that for a first-order linear operator $\oL \colon \Gamma( E) \to \Gamma( E)$ the principal symbol $\sigma_\oL\colon T^*\M \to \text{End}(E)$ can be characterized by $\oL( f u) = f \oL u + 	\sigma_\oL (d f )u$ where $u \in \Gamma(E)$ and $f \in C^\infty(\M)$.
If we choose local coordinates $(t, x^1, \ldots, x^n)$ on $\M$, with $x^i$ local coordinates on $\Sigma_t$, and a local trivialization of $E$, any linear differential operator $\oL\colon \Gamma(E) \to \Gamma(E)$ of first order reads in a point $p\in \M$ as
$$\oL\define  A_0(p) \partial_t + \sum_{j=1}^n A_j(p) \partial_{x^j} + B(p)$$
where the coefficients $A_0, A_j,B$ are $N\times N$ matrices, with $N$ being the rank of $E$, depending smoothly on $p\in\M$.
In these coordinates, Condition~(S) in Definition~\ref{def:symm syst} reduces to 
$$A_0=A_0^\dagger \qquad \text{and} \qquad A_j=A_j^\dagger$$
 for $j=1,\dots, n$. Condition~(P) reads as
\begin{align}\label{def:kappa}  \kappa\define \oL+\oL^\dagger = B +B^\dagger - \frac{\partial_t (\sqrt{g}A_0)}{\sqrt{g}}- \sum_{j=1}^n \frac{\partial_{x^j} (\sqrt{g}A_j)}{\sqrt{g}}>0,\end{align}
where $g$ is the absolute value of the determinant of the Lorentzian metric. Condition (H) can be stated as follows: For any future directed, timelike covector $\tau=dt + \sum_j \alpha_j dx^j$,  $$ \sigma_\oL(\tau)=A_0 + \sum_{j=1}^{n} \alpha_j A_j \quad \text{is positive definite.}$$

\begin{remark}\label{Dir on halfMink} 
With the above definition, we can immediately notice that the Dirac operator does not give rise to a symmetric system. Consider for example the half Minkowski spacetime $\mathscr{M}^4\define \RR^3\times[0,\infty)$ endowed with the standard metric $-dt^2+dx^2+dy^2+dz^2.$ In this setting the Dirac operator reads as
$$
\Dir=-\imath \gamma(e_0)\partial_t+ \imath  \gamma(e_1) \partial_{x} +\imath \gamma(e_2) \partial_{y} +\imath \gamma(e_3) \partial_{z}
$$
where $\gamma(e_i)$ are the Dirac matrices. By straightforward computation, we obtain 
\begin{itemize}
\item $(\i\gamma (e_j))^\dagger=\i\gamma(e_j)$ for $j\geq 1$ while $(-\i\gamma(e_0))^\dagger=\imath\gamma(e_0)$ which violates condition (S); 
\item $\kappa=0$ and hence condition (P) is violated;
\item $\sigma_\Dir (dt)= -\i \gamma(e_0)$ is not positive definite, therefore condition (H) is violated. 
\end{itemize}
 \end{remark}
 
Nonetheless, it still possible to find a fiberwise invertible endomorphism $Q \in \Gamma(\text{End}(S\M))$ such that  $Q \circ \Dir $ is  a symmetric hyperbolic system such that  for any compact subset of $\M$ the operator  $Q \circ \Dir + \lambda \Id$ is also a positive hyperbolic system for a suitable $\lambda >0$ as we will see below.

\begin{lemma}\label{Equivalence Dir and K}
Consider a globally hyperbolic spin manifold $\M$ with boundary $\bM$. Let $\Dir$ be the Dirac operator  and $e_0$, $\beta$ as in \eqref{def:e_0}. Then for $\lambda\in \mathbb R$ the first order differential operator $\oK_{\lambda}\colon \Gamma(S\M) \to \Gamma(S\M)$ defined by
\begin{align}\label{eq:K}\oK_{\lambda}\define\imath \gamma(e_0) \beta \Dir + \lambda \,
\Id\end{align}
is a symmetric hyperbolic system and its Cauchy problem 
\begin{equation}\label{CauchyK}
\begin{cases}{}
{\oK_{\lambda} }\Psi=\f \in \Gamma_{c}(S\M) \\
\Psi|_{\Sigma_0} = \mathfrak{h}\in \Gamma_{c}(S\Sigma_0) \\
{\oM}\Psi|_{\bM} =0
\end{cases} 
\end{equation}
is equivalent to the Cauchy problem  for the Dirac operator
\begin{equation}\label{CauchyDir2}
\begin{cases}{}
\Dir \psi=f \in \Gamma_c (S\M)  \\
\psi|_{\Sigma_0} ={h}\in \Gamma_c(S\Sigma_0) \\
\Mbc \psi|_{\bM} =0.
\end{cases}
\end{equation}
Moreover, for any compact set $\mathcal{R}\subset \M$, there exists a $\lambda>0$ such that $\oK_{\lambda}$ is a symmetric positive hyperbolic system. 
\end{lemma}
\begin{proof}
Since $\imath\gamma(e_0)\beta\Dir$ is a symmetric hyperbolic system, see e.g. \cite[Chapter~3]{Nicolas}, $\oK_{\lambda}$ is symmetric and hyperbolic for any $\lambda\in\RR$ and any positive smooth function $\beta$. Next we verify that the Cauchy problem for $\oK_{\lambda}$ and for $\Dir$ are equivalent for any $\lambda\in \RR$. We set $\Psi\define e^{-\lambda t} \psi$, which implies $\h \define e^{- \lambda t} h$, and $\f \define\imath e^{-\lambda t} \beta \gamma(e_0) f$ and obtain using \eqref{def:e_0}
\begin{align*}
\oK_{\lambda}\Psi=\oK_{\lambda}(e^{-\lambda t} \psi) = (\imath\gamma(e_0)\beta \Dir + \lambda\Id)(e^{-\lambda t}\psi) = \imath e^{-\lambda t} \gamma(e_0)\beta  \Dir \psi = \imath e^{-\lambda t} \beta \gamma(e_0) f.\end{align*}
Moreover, since $e^{-\lambda t} \neq 0$ for all $t\in\RR,$ we obtain the inverse map $(\f, \h) \mapsto (f,h)$ and we have
$$\oM \Psi|_\bM =   e^{\i\lambda t} \Mbc  \psi|_{\bM} = 0 \quad \text{if and only if} \quad \Mbc\psi|_\bM =0.
$$
This gives the equivalence of the two problems.\medskip 

Let now $\mathcal{R}\subset \M$ be compact. It remains to check the positivity condition (P) for $\oK_\lambda$. The operator $\kappa$, defined in \eqref{def:kappa} for an arbitrary  first order operator, is a zero order operator. Therefore, for any compact set $\mathcal{R}$, $\kappa|_{\Gamma(S\mathcal{R)}}$ is bounded and there exists a suitable $\lambda$ such that $\kappa$ is positive definite on $\mathcal{R}$. 
\end{proof}

\begin{example}\label{example2}
Let us consider the half Minkowski spacetime $\mathscr{M}^4$ and the Dirac operator
$$
\Dir=-\imath \gamma(e_0)\partial_t+ \imath  \gamma(e_1) \partial_{x} +\imath \gamma(e_2) \partial_{y} +\imath \gamma(e_3) \partial_{z} 
$$
from Remark~\ref{Dir on halfMink}. We can see that the operator 
$$\oK_{\lambda}=\imath \gamma(e_0)\Dir + \lambda\Id= \partial_t -\gamma(e_0)\gamma(e_1)\partial_x -\gamma(e_0)\gamma(e_2)\partial_y-\gamma(e_0)\gamma(e_3)\partial_z+ \lambda\Id$$ 
is a \emph{symmetric positive hyperbolic system} for any $\lambda>0$ on account of
\begin{align*}
 \big(\gamma(e_0)\gamma(e_i)\big)^\dagger= \gamma(e_i)^\dagger \gamma(e_0)^\dagger = - \gamma(e_i)\gamma(e_0) =  \gamma(e_0)\gamma(e_i) \quad \text{ for  $i=1,2,3$, and}
\end{align*}
 \end{example}
$\kappa=\oK_\lambda+\oK_\lambda^\dagger=2  \lambda \Id_{4\times 4}$ and $\sigma_\oK_{\lambda}(dt)=\Id_{4\times4}$.

\begin{remark}\label{rem:pos}
We note that one does not necessarily has to work with symmetric \emph{positive} hyperbolic systems but symmetric hyperbolic systems probably would suffice if one uses a slightly different approach. We decided to require positivity since we go for a energy inequality as in Lemma~\ref{lemma: energy estimates} which is easy to obtain and enough for our purpose. One probably also can obtain energy inequalities as in \cite[Thm. 5.3]{Ba} without the positivity.  The positivity requirement is also the reason why we work mostly on  time strips and/or compact subsets. But at the end by uniqueness of solutions we will still obtain a global solution in Section~\ref{sec:gs}.
\end{remark}

\section{Local well-posedness of the Cauchy problem}

 Let $\Dir$ be the Dirac operator on our globally hyperbolic spin manifold with timelike boundary $\bM$.  We denote by $\T$ the \textit{time strip} given by
$$\T\define t^{-1}([0,T])$$
where $T>0$ and $t\colon \M\to \RR$ is the chosen Cauchy time function (For negative times see Remark~\ref{rem-negT}). Let $\mathcal{O}$ be a compact subset of $\Sigma_T$.
 Let $\lambda>0$ be such that the operator $\oK_{\lambda}\colon \Gamma(S\M) \to \Gamma(S\M)$ defined as in \eqref{eq:K} is a symmetric positive hyperbolic system on 
$$\R\define J^-( \mathcal{O})\cap \mathcal{T}$$
 where  $J^-(\mathcal{O})$ (resp. $J^+(\mathcal{O})$) denotes the \textit{past set} (resp. \emph{future set}), namely the set of all points that can be reached by past-directed resp. future-directed causal curves emanating from a point in $\mathcal{O}$, cp. Figure~\ref{fig1}. This is always possible by Lemma~\ref{Equivalence Dir and K} and since  $\R\subset \M$ is compact. For the reason why we choose $\R$ as above compare below---especially Theorem~\ref{thm:Weak existence}.\medskip 

 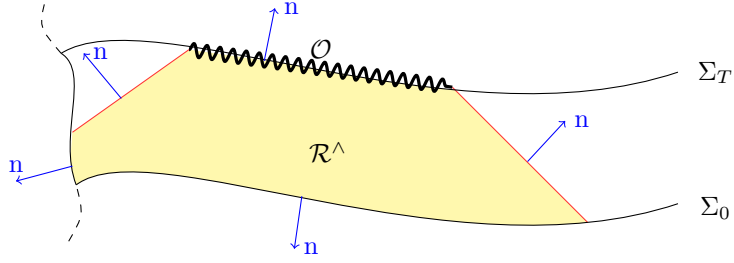
\begin{figure}
 \begin{tikzpicture}
\begin{scope}    
    \clip {(9,1.75) .. controls (6,0.75) and (2,2.75).. (0.8,2) -- (0.8,2)  .. controls (1.2,1.75)  and (0.75,0.9) .. (1.0,0.25)  -- (1.0,0.25) .. controls (2,1) and (6,-1) .. (9,0) };     
    \fill[yellow!40] { (0.9,0.3) -- (0.95,0.95) -- (2.5,2.05) -- (6,1.55) -- (7.8,-0.25) --(0.9,0.-3)-- (0.9,0.3) };      
\end{scope}

\draw[red!80] (2.5,2.05) -- (0.95,0.95);
\draw[red!80] (6,1.55) -- (7.8,-0.25);

 \draw[]  (1.0,0.25) .. controls (2,1) and (6,-1) .. (9,0) ; 
 \draw[]  (0.8,2) .. controls (2,2.75) and (6,0.75) .. (9,1.75); 
  \draw[]  (1.0,0.25) .. controls (0.75,0.9) and (1.2,1.75) .. (0.8,2); 
    \draw[dashed]  (0.9,-0.5) .. controls (1.1,-0.15) .. (1.0,0.25); 
     \draw[dashed]  (0.8,2) .. controls   (0.5,2.35) .. (0.6,2.75); 

  \draw[decorate, decoration=snake, segment length=5,very thick] (2.5,2.05)  -- ( 6,1.55)  node[midway,above,rotate=-1] {$\mathcal{O}$}  ; 

            \node  at (4.35,0.3) [label={$\R$}] { };
       \node  at (9.5,-0.45) [label={$\Sigma_0$}] { };
        \node  at (9.5,1.35) [label={$\Sigma_T$}] { };

\draw[color=blue, ->] (7,0.56)--(7.5,1.1) node[right] {n};
\draw[color=blue, ->] (3.5,1.9)--(3.64,2.6) node[right] {n};
\draw[color=blue, ->] (4,0.1)--(3.9,-0.6) node[right] {n};
\draw[color=blue, ->] (1.6,1.4)--(1.1,2) node[right] {n};
\draw[color=blue, ->] (0.95,0.5)--(0.2,0.3) node[above] {n};
    \end{tikzpicture}
    \caption{The set $\R =J^-(\mathcal{O})\cap \T$.}\label{fig1}
    \end{figure}

    In order to show existence of weak solutions and uniqueness of strong solutions for the Dirac Cauchy problem~\eqref{CauchyDir2}, we first shall derive  so-called \textit{``energy inequalities''}. These estimates have a clear physical consequence as we shall see in Proposition~\ref{prop:finite}: Any solution can propagate with at most speed of light. 

 \subsection{Energy inequalities}\label{sec:en}   

 We note that our energy inequality (even when considered without boundary) differs from the one considered in \cite[Thm. 5.3]{Ba}, see also Remark~\ref{rem:pos}. Such an energy inequality therein should also be obtainable in our setting, but since our approach to well-posedness goes over positive symmetric hyperbolic systems the following inequality serves our purpose.

\begin{lemma}\label{lemma: energy estimates} Let $\mathcal{O}$ be a compact subset of $\Sigma_T$ and choose $\lambda$  such that $\oK_{\lambda}$ is a symmetric positive hyperbolic system on $\R$ as above. Then there exists a constant $c>0$ such that for all $\Psi\in \Gamma(S\T)$ satisfying $\Psi|_{\Sigma_0}=0$ and $\oM\Psi|_{\bM}=0$ the \emph{energy inequality} 
\begin{equation}\label{Energy Inequality}
\|\Psi\|_{L^2(\R)}  \leq c \| \oK_{\lambda} \Psi\|_{L^2(\R)}
\end{equation}
holds. 
\end{lemma}

\begin{proof}
We first derive the Green formula for $\oK_\lambda$. In a local orthonormal frame $e_\mu$ we have $\oK_\lambda-\lambda = \sum_\mu \i \gamma(e_0)\beta \i \epsilon_\mu \gamma(e_\mu)\nabla_{e_\mu}=- \sum_\mu \gamma(e_0)\beta  \epsilon_\mu \gamma(e_\mu)\nabla_{e_\mu}$.
In a point of $\R$ we obtain using $\gamma(e_0)^\dagger=\gamma(e_0)$ and $(\gamma(e_0)\gamma(e_j))^\dagger=\gamma(e_0)\gamma(e_j)$ that
\begin{align*}
\<(\oK_\lambda &- \lambda)\Psi,\Psi\> =\<- \sum_\mu\beta\gamma(e_0)\epsilon_\mu \gamma(e_\mu) \nabla_{e_\mu}\Psi,\Psi\>\\=&\<(\nabla_{e_0}-\sum_j  \gamma(e_0)\gamma(e_j)\nabla_{e_j})\Psi, \beta\Psi\>\\
 =&e_0\<\Psi, \beta\Psi \> - e_j\<\Psi, \beta\gamma(e_0)\gamma(e_j)\Psi \>+\<\Psi, -\nabla_{e_0}(\beta \Psi) +\sum_j \nabla_{e_j}(\beta\gamma(e_0)\gamma(e_j)\Psi)\>\\
 =&  \sum_\mu e_\mu\<\Psi,- \beta\epsilon_\mu\gamma(e_0)\gamma(e_\mu)\Psi \>+\<\Psi,(\oK_\lambda^\dagger-\lambda)\Psi\>
 \end{align*}
 where $\oK_{\lambda}^\dagger$  is the formal $L^2$-adjoint of $\oK_{\lambda}$. Let $\n$ be the outward normal vector to $\bR$, cp. Figure~\ref{fig1}. 
 Thus, using Stokes formula and $\n=\sum_\mu \epsilon_\mu g(\n, e_\mu)e_\mu$ we obtain 
\begin{equation}\label{Green identity}
\scalar{\Psi}{ \oK_{\lambda} \Psi }_\R - \scalar{\oK_{\lambda}^\dagger \Psi }{ \Psi }_\R   =  \scalar{\Psi}{\beta \gamma(e_0) \gamma(\n)\Psi }_{\bR}   
\end{equation}
where  $\scalar{\cdot}{\cdot}_{\bR}$ is the induced $L^2$-product on ${\bR}$. 
Subtracting $ 2 \scalar{\Psi }{ \oK_{\lambda} \Psi}_\R$, taking the real part and  using that $\oK_{\lambda}$ is a symmetric positive system we thus obtain
\begin{equation}\label{passaggio1}
\begin{aligned}
\Re \scalar{\Psi}{ \beta \gamma(e_0)  \gamma(\n)\Psi }_{\partial \R } - 2 \Re \scalar{\Psi}{ \oK_{\lambda}\Psi}_\R  &=-\scalar{\Psi}{ (\oK_{\lambda} +\oK_{\lambda}^\dagger )\Psi}_\R  \\
& \leq   - 2c\scalar{\Psi}{ \Psi}_\R \,,
\end{aligned}
\end{equation}
for some $c>0$ and where $\Re$ denotes the real part. Next, let us decompose the boundary $\bR$ as 
$$\bR=\mathcal{O} \cup \Big( \Sigma_{0}\cap J^-(\mathcal{O})\Big) \cup Y\,,$$
where $Y\define  \partial J^-(\mathcal{O})\cap \text{interior}(\mathcal{T})$ is the boundary of the light cone inside the time strip $\T$.
The boundary term on $\Sigma_{0}\cap J^-(\mathcal{O})$ vanishes by assumption on $\Psi$. Hence, we can have non zero boundary contributions only at $\mathcal{O}$ and $Y$. To deal with these terms, we decompose  the boundary $Y$ further as
$$Y= (Y\cap \bM) \sqcup \big(Y \setminus (Y\cap \bM) \big).$$
Notice that if $\bR\cap\bM =\emptyset$ then $Y$ reduces to the boundary of the light cone $J^-(\mathcal{O})$ in the interior of $\T$. Imposing the boundary condition $\oM \Psi|_\bM=0$, we have  by condition~\eqref{general bound cond} that 
$$\scalar{\Psi}{\beta  \gamma(e_0) \gamma( \n)\Psi}_\bM =0\,. $$
Therefore, the boundary term on $Y\cap \bM$ vanishes. Hence, \eqref{passaggio1} reduces to 
\[2\Re\scalar{\Psi}{ \oK_{\lambda} \Psi }_\R\! -\! 2c\scalar{ \Psi }{ \Psi }_\R \!  \geq\!  \Re\scalar{\Psi}{\!\beta \gamma(e_0)\gamma( \n)\Psi }_{\mathcal{O} } + \Re\scalar{\Psi}{\!\beta \gamma(e_0)\gamma( \n)\! \Psi }_{Y \setminus (Y\cap \bM) }.\]
Let us remark that the right hand side of the latter equation is non negative definite: Indeed, since  $\mathcal{O}$ is a spacelike hypersurface,  $\n|_{\mathcal{O}}$ is future directed timelike and by hyperbolicity of $\oK_\lambda$,  $\fiber{.}{\beta \gamma(e_0)\gamma(\n) . }_{q}=\fiber{.}{\sigma_{\oK_\lambda}(\n^\flat).}{.}_q$  is a positive inner product for all $q\in \mathcal{O}$.
By continuity, also the contribution on  $Y \setminus (Y\cap \bM)$ is still positive semidefinite since  $Y \setminus (Y\cap \bM)$ is a lightlike hypersurface. Hence, together with  Inequality~\eqref{passaggio1} we obtain
$$2c\scalar{\Psi}{ \Psi}_\R \leq  2 \Re\scalar{\Psi}{ \oK_{\lambda}\Psi}_\R.$$
Thus, for all $\Psi\in  \Gamma(S\T)$ with $\Psi|_{\Sigma_0}=0$ and $\oM\Psi|_{\partial \M}=0$ we have  using the Hölder inequality that
\[\|\Psi\|_{L^2(\R)} \leq c^{-1} \|\oK_{\lambda}\Psi \|_{L^2(\R)}.\qedhere\]
\end{proof}

We conclude this section by giving an energy inequality for the formal $L^2$-adjoint of $\oK_{\lambda}$ analogous to~\eqref{Energy Inequality}. For any compact subset $\mathcal{O}'\subset \Sigma_0$, we restrict to the set $\Rv$  defined by
\begin{equation}\label{defRv}
\Rv\define \mathcal{T}\cap J^+(\mathcal{O}')\,.
\end{equation}
Using the same arguments as in the proof of Lemma~\ref{lemma: energy estimates}, we can conclude an analog energy inequality for the adjoint. Before stating this analog, let us shortly remark on the adjoint boundary condition.
\begin{remark}\label{rem:ad}
The formal adjoint $\oK_{\lambda}^\dagger$ of $\oK_{\lambda}$ is by definition given by
 \[ \scalar{\Phi}{\oK_{\lambda} \Psi}_{S\M}=\scalar{\oK_{\lambda}^\dagger \Phi}{\Psi}_ {S\M}
 \]
 for all $\Phi\in \Gamma_c(S\M)$ and $\Psi\in \Gamma_{c}(S\M)$. Let $\oK_{\lambda}^*$ be the adjoint of $\oK_{\lambda}$ where the domain of $\oK_{\lambda}$ is given by $\text{dom}\, \oK_{\lambda}\define \{ \Phi\in \Gamma_{c}(S\T)\ |\ \oM \Phi|_{\partial \M}=0, \Phi|_{\Sigma_0}=0\}$. By the Green identity~\eqref{Green identity} we have
 \begin{align*} \text{dom}\, \oK_{\lambda}^*= \Big\{\Phi\in \text{dom}\, (\oK_{\lambda}^\dagger)_{\text{max}}\, \Big|\, \int_{(\partial \M \cap \T) \cup \Sigma_T}\! \fiber{\Phi}{\beta\gamma(e_0)\gamma(\n)\Psi}=0\  \forall \Psi\in \text{dom} \oK_{\lambda}\Big\}
 \end{align*}
 where $\text{dom}\, (\oK_{\lambda}^\dagger)_{\text{max}}$ is the maximal domain of $\oK_{\lambda}^\dagger$.
 Note that 
 $$\left\{\Phi\in \Gamma_c(S\T)\ \Big|\ \Phi|_{\Sigma_T}=0,\ \int_{\partial \M \cap \T} \fiber{\Phi}{\beta\gamma(e_0)\gamma(\n)\Psi}=0\  \forall \Psi\in \text{dom} \oK_{\lambda}\right\}$$ is a core of $ \oK_{\lambda}^*$.  We set $(1-\oM^\dagger)\colon S_q (\T\cap \partial \M)\to S_q (\T\cap \partial \M)$ to be the orthogonal projection to 
$$\{\Phi\in S_q(\T\cap \partial \M)\ |\ \fiber{\Phi}{\gamma(e_0)\gamma(\n)\Psi}_q=0\ \forall \Psi\in  S_q(\T\cap \partial \M) \text{ with } \oM\Psi=0\}\,.$$
 We call $\oM^\dagger$ the \emph{adjoint boundary condition}, compare \cite[Sect.~7.2]{BB} for the analog notations in the elliptic case.  Note that as a matrix $M^\dagger$ is \emph{not necessarily} the hermitian adjoint of $M$ ($M$ can be multiplied by any invertible complex operator to obtain the same boundary condition.). 
Using this notation, we see that  
\begin{align*}\{\Phi\in \Gamma_c(S\T)\ |\ \Phi|_{\Sigma_T}=0,\ \oM^\dagger \Phi=0 \}\end{align*} is a core of $\oK_{\lambda}^*$.\medskip 

Note that for $\oM=\gamma(\n)-\imath$ being the MIT boundary condition then $\oM^\dagger=\oM$ --- see also Remark~\ref{rem:bd2}.
\end{remark}

\begin{lemma}\label{lemma: energy estimates2}
Let $\Phi\in \Gamma(S\T)$ satisfying $\Phi|_{\Sigma_T}=0$ and $\oM^\dagger\Phi|_{\bM}=0$. Then $\Phi$ satisfies the inequality
\begin{equation}\label{Energy Inequality2}
\|\Phi\|_{L^2(\mathcal{\Rv})}  \leq C \| \oK_{\lambda}^\dagger \Phi\|_{L^2(\mathcal{\Rv})}\,.
\end{equation}
\end{lemma}

\subsection{Uniqueness and finite propagation speed}\label{sec:fps}

We are now ready to see by standard arguments that if there exists a smooth solution to the Cauchy problem~\eqref{CauchyK}, then it propagates with at most speed of light.

\begin{proposition}[Finite speed of propagation]\label{prop:finite} 
Any smooth solution to the Dirac Cauchy problem~\eqref{CauchyDir} propagates with at most the speed of light, i.e., its support  is inside the region 
$$ \mathcal{V}\define \Big( J^+\big(\supp\, f \big)\cup J^+(\supp h)\Big),$$ 
see Figure~\ref{fig:finite speed}.
\end{proposition}

\begin{proof} For any time strip $\T$, consider any point $p$ outside the region $\mathcal{V}\cap \T$.
\begin{figure}
\begin{tikzpicture}
\begin{scope}       
    \clip {(9,1.75) .. controls (6,0.75) and (2,2.75).. (0,1.75) -- (0.8,2)  .. controls (1.2,1.75)  and (0.75,0.9) .. (1.0,0.25)  -- (1.0,0.25) .. controls (2,1) and (6,-1) .. (9,0) };    
    \fill[yellow!30] { plot[smooth, tension=1] coordinates { (4.5,0.8) (5,0.58) (6,0.25) (6.6,1) (6,2) (4.5, 1.5)   (4.5,0.8) }};
;
                \node  at (5.5,1) {$\supp \f$};
\end{scope}

\begin{scope}    
    \clip {(9,1.75) .. controls (6,0.75) and (2,2.75).. (0.8,2) -- (0.8,2)  .. controls (1.2,1.75)  and (0.75,0.9) .. (1.0,0.25)  -- (1.0,0.25) .. controls (2,1) and (6,-1) .. (9,0) };     
    \fill[yellow!30] { (2,0.4) -- (0.97,1.45) -- (0.8,2) -- (0.8,2) .. controls (1,3).. (7.4,1.47)  -- (6.38,0.4) --  (4.7,0.7)  -- (3.99,0.1) -- (2,0.27) };
;
                \node  at (5.5,1) {$\supp \f$};
\end{scope}

\begin{scope}        
    \clip {(0,0) .. controls (2,1) and (6,-1) .. (9,0) -- (9,1.75) -- (9,1.75) .. controls (6,0.75) and (2,2.75).. (0,1.75) -- (0,0)};    
    \fill[blue!30] { (7.8,1.15) -- (7.1,-0.3)  .. controls ( 7.5, -.5) .. (8.7,-0.1) --(7.8,1.15)  };
;
                \node  at (5.5,1) {$\supp \f$};
\end{scope}

 \draw[]  (1.0,0.25) .. controls (2,1) and (6,-1) .. (9,0) ; 
 \draw[]  (0.8,2) .. controls (2,2.75) and (6,0.75) .. (9,1.75); 
  \draw[]  (1.0,0.25) .. controls (0.75,0.9) and (1.2,1.75) .. (0.8,2); 
    \draw[dashed]  (0.9,-0.5) .. controls (1.1,-0.15) .. (1.0,0.25); 
     \draw[dashed]  (0.8,2) .. controls   (0.5,2.35) .. (0.6,2.75); 

  \draw[decorate, decoration=snake, segment length=5,very thick] (2,0.38)  -- ( 4,0.1)  node[midway,below,rotate=-5] {$\,\,\supp \h$}  ; 

            \draw[]  plot[smooth, tension=1] coordinates { (4.5,0.8) (5,0.58) (6,0.25) (6.6,1) (6,2) (4.5, 1.5)   (4.5,0.8) };

\draw[red!80] (2,0.4) -- (0.97,1.45);
\draw[red!80] (3.99,0.1) -- (4.7,0.7);
\draw[red!80] (6.38,0.4) -- (7.4,1.47);

\draw[blue!80] (7.8,1.15) -- (7.1,-0.3);
\draw[blue!80] (7.8,1.15) -- (8.7,-0.1);

            \node  at (8,0.85) [label={$p$}] {};	
            \node  at (7.8,1.15) {$\mathsmaller{\mathsmaller{\bullet}}$};	                                    
            \node  at (3.5,0.35) [label={$q$}] { };
            \node  at (3.3,0.65) {$\mathsmaller{\mathsmaller{\bullet}}$};
            \node  at (2.25,1) [label={${\mathcal{V}\cap \T}$}] { };
            \node  at (7.85,-0.1) [label={$\R$}] { };
            \node  at (9.5,-0.45) [label={$\Sigma_0$}] { };
            \node  at (9.5,1.35) [label={$\Sigma_T$}] { };

    \end{tikzpicture}
    \caption{Finite propagation of speed -- $\mathcal{V}\cap \T$.}
    \label{fig:finite speed}
\end{figure}
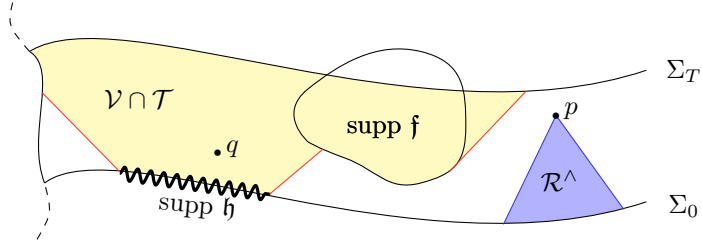
Then there exists a $\lambda$ such that $\oK_{\lambda}$ is a symmetric  positive hyperbolic system on $\R=\mathcal{T}\cap J^-(p)$.
Let $\Psi=e^{-\lambda t}\psi$. Then, by Lemma~\ref{Equivalence Dir and K} $\Psi$ is a solution to \eqref{CauchyK} with $\f=\imath e^{-\lambda t} \beta \gamma(e_0)f$ and $\h= e^{-\lambda t}h$. 
By Lemma~\ref{lemma: energy estimates} and $\f|_\R\equiv 0$, $\h|_{\R\cap \Sigma_0}\equiv 0$, $\Psi$ vanishes in $\R$. Hence, $\Psi$ vanishes outside $\mathcal{V}$.

The finite propagation of speed for  a smooth solution of the Dirac Cauchy problem~\eqref{CauchyDir2} then follows by Lemma~\ref{Equivalence Dir and K}.
\end{proof}

\begin{proposition}[Uniqueness]\label{prop:unique}
Suppose there exist $\Psi,\Phi \in \Gamma(S\T)$ satisfying the same Cauchy problem~\eqref{CauchyK} for some $\f\in \Gamma_{cc}(S\M)$ and $\h\in \Gamma_{cc}(S\Sigma_0)$. Then $\Psi=\Phi$.
In particular, this also gives uniqueness of smooth solutions for the Dirac Cauchy problem~\eqref{CauchyDir2}.
\end{proposition}
\begin{proof}
Since $\Psi$ and $\Phi$ satisfy the same initial-boundary value problem~\eqref{CauchyK}, then $\Psi-\Phi\in \Gamma(S\T)$ is a solution of \eqref{CauchyK} with $\f=0$ and $\h=0$.
By Proposition~\ref{prop:finite}, the supports of $\Psi$ and $\Phi$ are contained in $\R$ for $\mathcal{O}\define \mathcal{V}\cap \Sigma_T$. Therefore, we can use Lemma~\ref{lemma: energy estimates} to conclude that $\Psi-\Phi$ is zero.\medskip

The uniqueness of the Dirac Cauchy problem~\eqref{CauchyDir2} then follows by Lemma~\ref{Equivalence Dir and K}.
\end{proof}

\subsection{Weak solutions in a time strip - definition}
\label{sec:ws}

In the following we will focus on the Cauchy problem  \eqref{CauchyK} for  $\h\equiv 0$. This can be done without loss of generality, since:

\begin{remark}
For any Cauchy problem with any nonzero initial data $\mathfrak{h}\in \Gamma_{cc}(S\Sigma_0)$ there exists an equivalent Cauchy problem with zero initial data, namely 
\begin{equation*}
\begin{cases}{}
\oK_{\lambda} \Psi=\f  \\
\Psi|_{\Sigma_0} = \mathfrak{h} \\
\oM\Psi|_{\bM} =0
\end{cases} \, \Longleftrightarrow \quad \begin{cases}{}
\oK_{\lambda} \tilde\Psi=\tilde{\mathfrak{f}} \\
\tilde\Psi|_{\Sigma_0} = 0  \\
\oM \tilde\Psi|_{\bM} =0
\end{cases} 
\end{equation*}
{for any $\f,\tilde\f\in\Gamma_{c}(S\M)$}. Here $\tilde{\mathfrak{f}}(t,{x})\define \f(t,{x}) - \lambda\mathfrak{h}({x}) - \imath\beta\gamma(e_0) \Dir \mathfrak{h}$ and $\tilde{\Psi}(t,{x})\define \Psi(t,{x})-\mathfrak{h}({x})$. Note that $\f\in\Gamma_{cc}(S\M)$ if and only if $f\in\Gamma_{cc}(S\M)$.
\end{remark}

With the help of the Energy inequality~\eqref{Energy Inequality} we shall prove the existence of a \textit{weak solution} for the mixed initial-boundary value problem~\eqref{CauchyK} for $\oK_{\lambda}$ as in~\eqref{eq:K} (for fixed  $\f\in\Gamma_{cc}({S\M})$ and $\h\equiv 0$).
To this end, let $\lambda$ be such that $\oK_{\lambda}^\dagger$ is a symmetric positive hyperbolic system on $\Rv=\T\cap J^+(\mathcal{O}')$ for 
\begin{equation}\label{def:V}
\mathcal{O}'\define J^-(\V\cap \Sigma_T) \cap \Sigma_0 \text{ with }\mathcal{V}\define J^+\big(\supp\,\f \big) \cap \mathcal{T}\,.
\end{equation}
We denote by  
\[\mathscr{H} \define \overline{\big({\Gamma_c(S\T)},{\scalar{.}{.}}_{\T} \big)}^{\scalar{.}{.}_{\T}} \]
 the $L^2$-completion of  $\Gamma_c(S\T)$.

\begin{definition}\label{def:Weak existence}
We call $\Psi\in\cH$ a \emph{weak solution} to the Cauchy problem~\eqref{CauchyK} restricted to an open subset $\mathcal{U}$ of $\T$ with $\f\in \Gamma_{cc}(S\M)$ and $\h\equiv 0$ if the relation
\begin{equation}\label{scalprod}
\scalar{\Phi}{\f}_{\mathcal{U}}=\scalar{\oK_{\lambda}^\dagger \Phi}{\Psi}_{\mathcal{U}}
\end{equation}
holds for all  $\Phi \in \Gamma_{c}(S\mathcal{U})$ satisfying $\oM^\dagger\Phi|_{\bM }=0$  and $\Phi|_{\Sigma_T}\equiv 0$.
\end{definition}
In order to check that this is the right definition let us give the following remark.
\begin{remark}
 Let $\Psi$ be a weak solution as defined above that is even smooth. Then by testing first only with $\Phi\in \Gamma_{cc}(S\mathcal{U})$ we immediately obtain with the Green identity that $\oK_{\lambda} \Psi=\f$. Using this in \eqref{scalprod} with the Green identity for general $\Phi\in \Gamma_{c}(S\mathcal{U})$ gives 
 \[ 0= \scalar{\Phi}{\beta \gamma(e_0)\gamma(\n)\Psi}_{(\text{interior}(\T)\cap \partial\M) \cup \Sigma_0\cup \Sigma_T}. \]
 The part on $\Sigma_T$ vanishes since $\Phi_{\Sigma_T}\equiv 0$. Moreover, using test functions $\Phi$ that have support near $\Sigma_0$ resp. $\partial \M\cap \text{interior}(\T)$ one sees as in Remark~\ref{rem:ad}  that $\Psi_{\Sigma_0}=0$ resp. $\oM \Psi|_{\partial \M}=0$.
\end{remark}

\begin{remark}
 Analogously, one sees that a spinor $\Psi\in\cH$ is a weak solution to \eqref{CauchyK} restricted to an open subset $\mathcal{U}$ of $\T$ with $\h\in \Gamma_{cc}(S\Sigma_0)$  if the relation
\begin{equation*}
\scalar{\Phi}{\f}_{\mathcal{U}}=\scalar{\oK_{\lambda}^\dagger \Phi}{\Psi}_{\mathcal{U}} +\scalar{\Phi}{\beta \gamma(e_0)\gamma(\n)\h}_{\Sigma_0\cap \mathcal{U}}
\end{equation*}
holds for all  $\Phi \in \Gamma_{c}(S\mathcal{U})$ satisfying $\oM^\dagger\Phi|_{\bM }=0$  and $\Phi|_{\Sigma_T}\equiv 0$.
\end{remark}

\subsection{Differentiability of  weak solutions}\label{diffsol}

Before examining existence of a weak solution, we want to show that a weak solution, if it exists,  is a strong solution and, in particular, smooth. Since this is a local question, we use the theory for hyperbolic systems on subsets of $\mathbb R^{n+1}$---in particular  that a weak solution is a (semi-)strong solution, \cite[Section 1]{LP},  and the regularity estimates for strong solutions in \cite[Theorem~3.1]{RM}. Since the definition of strong solution in the sense Lax--Phillips slightly differs with the one given by Rauch--Massey, we will denote it by semi-strong in the following:

\begin{definition}\label{def:strong sol}
Let $U\subset\M$ be a compact subset in $\M$.
We say that $\Psi\in \mathscr{H}$ is a \emph{semi-strong solution} of the initial-boundary value problem~\eqref{CauchyK} if there exists a sequence of sections $\Psi_k \in W^{1,2}(\Gamma(SU))$ such that $\oM\Psi_k=0$ on $\partial\M \cap U$, $\Psi_k=0$ on $\Sigma_0$ and  
$$ \| \Psi_k - \Psi \|_{L^2(U)} \xrightarrow{k\to \infty} 0 \quad \text{ and } \quad \| \oK_{\lambda}\Psi_k - \f \|_{L^2(U)} \xrightarrow{k\to \infty} 0.$$ The solution is called \emph{strong} if additionally the sequence $\Psi_k$ can chosen to be smooth.
 \end{definition}

 We concentrate on points in the boundary $p\in  \partial \M$ (the other points will even be easier since we do not have to care about boundaries) and firstly define a convenient chart as follows, compare also Figure~\ref{fig:chart}:
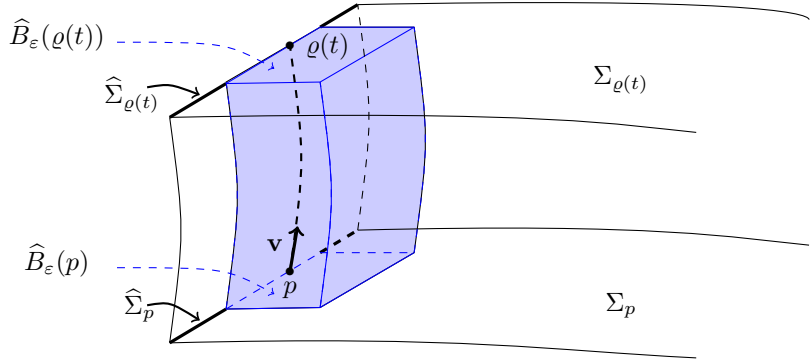
\begin{figure}
\begin{tikzpicture}

\filldraw[fill=blue!20!white] (0.75,0.45) --(2,0.47) -- (3.25,1.2) -- (3.25,1.2) .. controls (3.45,2.7) .. (3.25,4.2) -- (3.25,4.2) --(2,4.2) -- (0.75,3.47) .. controls (0.95,1.95) .. (0.75,0.45) ;

 \draw (0,0) .. controls (4,0.1) and  (5,-0.1) .. (7,-0.2) ; 
 \draw (2.5,1.5) .. controls  (6.5,1.6) and  (7.5,1.4) .. (8.5,1.3) ; 
\draw[very thick] (0,0) -- (0.75,0.45);
\draw[dashed, very thick] (2,1.2) -- (2.5,1.5);
\draw[blue,dashed ] (0.75,0.45) -- (2,1.2) ;
\draw[blue, ] (0.75,0.45) -- (2,0.47) ;
\draw[blue, dashed] (2,1.2) -- (3.25,1.2) ;
\draw[blue, ] (2,0.47) -- (3.25,1.2) ;

 \draw (0,3) .. controls  (5,3.1) and  (6,2.9)  .. (7,2.8) ;  
\draw (2.5,4.5)..controls (7.5,4.6) and (8.5,4.4) .. (8.5,4.3) ;  
\draw[very thick] (0,3) -- (0.75,3.45);
\draw[very thick] (2,4.2) -- (2.5,4.5);
\draw[blue,] (0.75,3.45) -- (2,4.2) ;
\draw[blue, ] (0.75,3.45) -- (2,3.47) ;
\draw[blue, ] (2,4.2) -- (3.25,4.2) ;
\draw[blue, ] (2,3.47) -- (3.25,4.2) ;

 \draw  (0,0) .. controls (0.2,1.5) .. (0,3); 
\draw[dashed]   (2.5,1.5) .. controls (2.7,3) .. (2.5,4.5); 

\draw[dashed,blue,]  (0.75,0.45) .. controls (0.95,1.95) .. (0.75,3.47); 
\draw[dashed,blue,]  (2,0.47) .. controls (2.2,1.97) .. (2,3.47); 
\draw[blue,]  (2,0.47) .. controls (2.2,1.97) .. (2,3.47); 
\draw[dashed,blue,]  (3.25,1.2) .. controls (3.45,2.7) .. (3.25,4.2); 

            \node  at (1.6,0.95) {$\mathsmaller{{\bullet}}$};	                                    
            \node  at (1.6,0.70) {$p$};	       
            \node  at (1.6,3.95) {$\mathsmaller{{\bullet}}$};	                                    
            \node  at (2.1,3.95) {$\varrho(t)$};	                                    
\draw[dashed,black, thick]  (1.6,0.95)  .. controls (1.8,2.45) .. (1.6,3.95); 
\draw[black, very thick,->]  (1.6,0.95)  -- (1.7,1.55); 
            \node  at (1.4,1.3) {$\bf v$};	       

            \node  at (6,0.5) {$\Sigma_p$};	       
            \node  at (6,3.5) {$\Sigma_{\varrho(t)}$};

    \draw[dashed,blue, ,->]  (-0.7,1) .. controls (0.75,1) .. (1.39,0.65) ; 
     \node  at (-1.5,1.1) {$\widehat{B}_\epsilon(p)$};	
   
    \draw[dashed,blue, ,->]  (-0.7,4) .. controls (0.75,4) .. (1.39,3.65) ; 
     \node  at (-1.5,4.1) {$\widehat{B}_\epsilon(\varrho(t))$};	
                   
            \node  at (-0.4,0.5) {$\widehat{\Sigma}_p$};	       
    \draw[black,thick ,->]  (-0.3,0.5)  .. controls (0,0.55) .. (0.4,0.3) ; 

            \node  at (-0.5,3.3) {$\widehat{\Sigma}_{\varrho(t)}$};	       
   \draw[black, thick,->]  (-0.3,3.4)  .. controls (0,3.45) .. (0.4,3.3) ; 

    \end{tikzpicture}
    \caption{Fermi coordinates on each Cauchy surface.}
    \label{fig:chart}
\end{figure} 
 Let ${\Sigma}_p$ be the Cauchy surface of $\M$ to which $p$ belongs to. For $q\in \partial \M$ let $\widehat{\Sigma}_q \define \Sigma_q\cap \partial \M$ be the corresponding Cauchy surface in the boundary. Let $\varrho\colon [0,\epsilon]\to \partial \M$ be the timelike geodesic in $\partial \M$ starting at $p$ with velocity ${\bf v} \in T_p\bM$ where ${\bf v}$ is a normalized, future-directed, timelike vector perpendicular to $\widehat{\Sigma}_p$ in $\partial \M$.  Let $\widehat{B}_\epsilon (\varrho(t))$ be the $\epsilon$-ball in $\hat{\Sigma}_{\varrho(t)}$ around $\varrho(t)$. On these balls we choose geodesic normal coordinates $\widehat{\kappa}_t\colon B_\epsilon^{n-1}(0)\subset \mathbb R^{n-1}\to \widehat{B}_\epsilon (\varrho(t))$. Moreover, inside each $\Sigma_{\varrho(t)}$ we choose Fermi coordinates with base $\widehat{B}_\epsilon (\varrho(t))$.
  Thus, we obtain a chart in $\Sigma_{\varrho(t)}$ around $\varrho(t)$ as \label{def:U_p}  
\begin{align*} 
\widetilde\kappa_t\colon B_\epsilon^{n-1}(0)\times [0,\epsilon]\subset \mathbb R^{n}&\to U_{\epsilon} (\widehat{B}_\epsilon (\varrho(t)))\define \{ q\in \Sigma_{\varrho(t)}\ |\ \mathrm{dist}_{\Sigma_{\varrho(t)}}(q, \widehat{B}_\epsilon (\varrho(t))\leq \epsilon\}
\\
 (y,z) & \mapsto \exp^{\perp, \Sigma_{\varrho(t)}}_{\widehat{\kappa}_t(y)}(z)
\end{align*}
where $\exp^{\perp, \Sigma_{\varrho(t)}}_{\widehat{\kappa}_t(y)}(z)$ is the normal exponential map in $\Sigma_{\varrho(t)}$ starting at $\widehat{\kappa}_t(y)$ with velocity perpendicular to $\widehat{\Sigma}_{\varrho(t)}=\partial \Sigma_{\varrho(t)}$ pointing in the interior and with magnitude $z$. Putting all this together we obtain a chart
  \begin{align*}
   \kappa_{p} \colon [0,\epsilon] \times B_\epsilon^{n-1}(0)\times [0,\epsilon] \subset \mathbb R^{n+1}&\to U_p\define \bigcup_{t\in [0,\epsilon]} U_\epsilon (\widehat{B}_\epsilon (\varrho(t))) \subset \M\\
   (t,y,\bar{z})&\mapsto \widetilde\kappa_t(y,\bar{z}).
  \end{align*}
    Note that sections of the spinor bundle $SU_p$ are now just vector-valued functions $U_p\to \mathbb C^N$ where $N$ is the rank of the spinor bundle.\medskip
    
     For more details on the above geometric maps compare e.g. \cite{GS}.     For us here, the only purpose of those charts is to specify coordinates such that near the point $p$ the Cauchy problem is near enough to the Minkowski standard form and  will take the form as in \cite{LP, RM}. To see this, let us  first consider the model case of a `general half' of the Minkowski space.
 
 \begin{example}\label{ex:mink}
  Let $\overline{\mathscr{M}}$ be the Minkowski space with coordinates  $x=(x^0, \ldots, x^n)$ with the standard foliation of Cauchy surfaces. We set $t\define x^0$, $y=(x^1,\ldots, x^{n-1})$ and $z\define x^n$. For $|a|<1$, the hypersurface $\N_a\define \{z=at\}$ is timelike and $\mathscr M_a=\{ z\geq at\}\subset \overline{\mathscr{M}}$ is a globally hyperbolic manifold with a timelike boundary. We use $(\overline{t}\define t,y,\overline{z}\define z-at)$ as new coordinates on $\mathscr M_a$. Then, $(\overline{t},y,\overline{z})\mapsto (\overline{t}, y, \overline{z}+a\overline{t})$ is exactly the map $\kappa_p$ from above for any $p\in \mathcal{N}_a$.\medskip
  
  Then, together with $\gamma(e_0)=\gamma(e_0)^{-1}$, we have \begin{align*}
  \imath\gamma(e_0)\Dir &= \partial_t -\sum_{j=1}^{n-1} \gamma (e_0)\gamma (e_j)\partial_{x^j}-\gamma ({e_0})\gamma ({e_n})\partial_{z}\\
  &= \partial_{\overline{t}} -  \sum_{j=1}^{n-1} \gamma ({e_0})\gamma ({e_j})\partial_{x^j}-(\gamma ({e_0})\gamma ({e_n})+a){\partial_{\overline{z}}}. 
  \end{align*}

  Thus,  $\oK_{\lambda}=\imath\gamma(e_0)\Dir+\lambda \mathrm{Id}$, as in Example~\ref{example2}, has the form
  \begin{equation}\label{eq:Pform}
   \oK_{\lambda}= \partial_{\overline{t}} + \sum_{j=1}^{n-1} A_j(x)\partial_{x^j} + A_{\overline{z}}(x) \partial_{\overline{z}} + B(x).
  \end{equation}
  with $A_{\overline{z}}(x)=-\gamma ({e_0})\gamma ({e_n})-a$.   Since $|a|<1$, $A_{\overline{z}}(x)$  is nonsingular on $\partial \mathscr M_a$. Hence, since $\ker\oM|_q$ varies smoothly with $q\in\partial \mathscr M_a$, after restricting to some cube in $\mathscr M_a$ we are exactly in the situation considered in \cite{LP, RM}.
 \end{example}

 \begin{corollary}\label{cor_Up}
  For the Dirac operator $\Dir$ on a globally hyperbolic spin manifold with timelike boundary $\bM$ and $p\in \partial \M$, there is a sufficiently small $\epsilon >0$ such that in the coordinates $\kappa_p$ from above there is an invertible operator $\oE\colon \Gamma(U_p, \mathbb C^N)\to \Gamma(U_p, \mathbb C^N)$ such that $\oE\oK_{\lambda}$ has the form \[\partial_{\overline{t}} + \sum_{j=1}^{n-1} A_j(\bar{t}, y, \bar{z})\partial_{x^j} + A_{\overline{z}}(\bar{t}, y, \bar{z}) \partial_{\overline{z}} + B(\bar{t}, y, \bar{z})\] with $A_{\bar{z}}$ nonsingular on the boundary. In particular, any weak solution of the Cauchy problem~\eqref{CauchyK} gives rise to a weak solution to the Cauchy problem
    \begin{equation}\label{CauchyK_2}
\left\{\begin{matrix}
\widehat{\oK_{\lambda}}\define \oE\oK_{\lambda} \Psi&=\f \in \Gamma_{c}(SU_p) \\
\Psi|_{V_p}& = \mathfrak{h}\in \Gamma_{c}(SV_p) \\
\oM\Psi|_{\bM} & =0 \qquad\qquad\quad
\end{matrix} \right.
\end{equation}
 on $U_p$ and vice versa. Here $V_p\define U_p\cap t^{-1}(0)$ and weak solution of \eqref{CauchyK_2} is defined analogously as in Definition~\ref{def:Weak existence}.
 \end{corollary}

 \begin{proof}
  By the choice of the coordinates,  the Dirac operator will look in $p$ exactly  as for the Minkowski space computed in Example~\ref{ex:mink} with $a=dt(n_p)$ where $n_p$ is the normal at $p$ of $\hat{\Sigma}_p$ in $\partial \M$ and $t$ is the global time function of $\overline \M$. Note that the role of $\N_a$ is taken by the tangent plane of $\partial \M\hookrightarrow \M$ in $p$ and that $|a|<\beta(p)^{-1}$.
  Since everything is continuous, we can find a sufficiently small $\epsilon>0$ such that there is an invertible linear map $\oE \colon \Gamma(U_p, \mathbb C^N)\to \Gamma(U_p, \mathbb C^N)$ with $\oE|_{S_p\M}=1$ and  such that $\oE\oK_{\lambda}$ has the required form with $A_{\overline{z}}(p)=-\beta \gamma(e_0)\gamma(e_n)-a$.
  
  Moreover, we obtain a weak solution of \eqref{CauchyK_2} as required by taking the weak solution of \eqref{CauchyK} where the right handside is given by $(\oE^{-1}\f, \h, 0)$.
 \end{proof}

\begin{lemma}[Locally strong solution]\label{lem:strong} 
A weak solution $\Psi$ of the Cauchy problem~\eqref{CauchyK_2} is a strong solution on $U_p$.
\end{lemma}

\begin{proof} 
 The last Corollary tells us that we can apply \cite[Section 2]{LP} in order to obtain the existence of a semi-strong solution, i.e., there is a sequence of continuous sections $\Phi_k\in W^{1,2}(U_p, \mathbb C^N)$ with $\oM\Phi_k|_{\partial \M\cap U_p}=0$ and $\Vert \Phi_k-\Psi\Vert_{L^2(U_p)}\to 0$ and $\Vert\widehat{\oK_{\lambda}}\Phi_k-\f\Vert_{L^2(U_p)}\to 0$. It remains to argue, that we can approximate the $\Phi_k$  by smooth $\Psi_k$ still fulfilling the boundary condition and the convergences from above. 
 
 This can be achieved using standard  Sobolev theory; we refer to \cite{Evans} for more details: First, choose $u_i\in \Gamma(U_p, \mathbb C^N)$, $i=1,\ldots, r$,  such that for each $q\in \partial \M\cap U_p$ they form a basis of $\ker \oM|_q$ and are linearly independent in all $q\in U_p$. Since $\M$ depends smoothly on the base point and has constant rank this is always possible.
Choose  $u_j\in \Gamma(U_p, \mathbb C^N)$, $j=r+1,\ldots, N$ such that $u_1(q), \ldots, u_N(q)$ is a basis of $\mathbb C^N$ at each $q\in U_p$. A section $\Phi\in W^{1,2}(U_p, \mathbb C^N)$ can now be expressed as $\Phi=\sum_{i=1}^N a_i u_i$ for $a_i\colon U_p\to \mathbb C$. We denote by $\Phi^+$ the part of $\Phi$ spanned by $u_1$ to $u_r$ and set $\Phi^-\define\Phi-\Phi^+$. Using the $a_i$ as the new coordinates, we decompose the solution $\Psi$ into $\Psi_k^+\in W^{1,2}(U_p, \mathbb C^r)$ and  
$\Psi_k^-\in W^{1,2}(U_p, \mathbb C^{N-r})$. Thus, there is a sequence $\Psi_{k,j}^+\in \Gamma(U_p, \mathbb C^r)$ that converges to $\Psi_k^+$ in $W^{1,2}$ and analogously a smooth sequence $\Psi_{k,j}^-$ converging to $\Psi_k^-$ in $W^{1,2}$. Moreover, by definition $\mathtt{tr}\, \Psi^+_k= 0$, where $\mathtt{tr}$ is the trace map $W^{1,2}(U_p, \mathbb C^r)\to L^2 (U_p\cap \partial \M, \mathbb C^r)$. Thus, $\Psi_{k,j}^+$ can be chosen to be zero on $U_p\cap \partial \M$. Thus, $\Psi_{k,j}=\Psi_{k,j}^++\Psi_{k,j}^-$, where we use the embeddings $\mathbb C^r\hookrightarrow \mathbb C^N$ and  $\mathbb C^r\hookrightarrow \mathbb C^{N-r}$ from above, are smooth sections fulfilling $\oM\Psi_{k,j}|_{U_p\cap \partial \M}=0$. Choosing a diagonal sequence we obtain smooth $\Psi_k$ approximating $\Psi$ as in Definition~\ref{def:strong sol}. 
\end{proof}

Next we want to see whether the strong solution on $U_p$  is actually smooth. For that we would like  
to use the result \cite[Theorem 3.1]{RM} by Rauch and Massey. For $p\in \Sigma_0$ this is immediate. But for $p\in \Sigma_t$ for $t>0$ the solution might now touch the boundary and compatibility issues occur. This is expectable since just assuming  $\f\in \Gamma_c(S\M)$ and $\h\in \Gamma_c(S\Sigma_0)$ (and not as in Theorem~\ref{maintheorem} compactly supported in the interior) is not sufficient to guarantee
that the solution of the Cauchy problem~\eqref{CauchyK} is smooth. 

\subsubsection{Compatibility conditions and smoothness of the solution}

To see the appearance of the compatibility issues let us start with the easiest example:

\begin{example}
Let $\mathscr M_a$ be the half Minkowski spacetime as described in Example~\ref{ex:mink} and consider the Cauchy problem~\eqref{CauchyK} $\oK_{\lambda} \Psi=0$, $\Psi|_{t=0} = \mathfrak{h}$ and $\oM\Psi|_{\overline{z}=0}= 0$. Assume that $\oM$ does not depend on $t$, that is (e.g.) true for MIT boundary conditions.  Suppose that $\Psi$ is $k$-differentiable. Set $\oH\define \partial_t-\oK_{\lambda}$. Then it satisfies
$$ 0 = \partial_t^k \big( \oM \Psi|_{\overline{z}=0} \big)|_{t=0} =  \big( \oM \partial_t^k \Psi|_{\overline{z}=0} \big)|_{t=0} = \big( \oM \oH^k \Psi(t)|_{\overline z=0} \big)|_{t=0} =  \oM \big(\oH^k \h\big)|_{\overline z=0} \,. $$
Therefore, any initial data have to satisfy a compatibility condition. 
 \end{example}

\begin{remark}\label{star}
 In the previous example we massively used that $\oM$ and $\oH$  does not depend on $t$. In the more general case of $\oK_{\lambda}$ for a Dirac operator on a globally hyperbolic manifold in the coordinates on $U_p$ defined on page~\pageref{def:U_p} this is in general not the case.   
 
  In order to obtain a general compatibility condition on $\f$ and $\h$ set  $\oH= {\partial_t}- \oE\oK_\lambda$  with $\oE$ as in Corollary~\ref{cor_Up} and 
\begin{align}\label{eq_hk} \h_k\define 
\sum_{j=0}^{k-1} 
\frac{(k-1)!}{j! (k-1-j)!}
\big( \partial^j_t  \oH \big)|_{V_p} \h_{k-1-j}+ \partial_t^k  \f|_{V_p}
\end{align}
for all $k\geq 1$ with $\h_0=\h$, $V_p=U_p\cap \Sigma_0$.  Impose that the data $\h\in \Gamma_c(S\Sigma_0)$ and $\f\in\Gamma(S\M$) satisfy
$$ \sum_{j=1}^k 
 \frac{k!}{(k-j)!}
( \partial_t^j \oM )|_{V_p}\h_{j-1} =0.
$$  
Translating this back for our Dirac Cauchy problem ~\eqref{CauchyDir2} in the Hamiltonian form where $\Ham\define \partial_t-\imath\beta \gamma(e_0)\Dir$ 
$$(\partial_t - \Ham )\psi = \imath \beta \gamma(e_0) f$$
the compatibility condition for $h\in\Gamma_c(S\Sigma_0)$ and $\f\in\Gamma(S\M)$ reduces to
$$ \sum_{j=1}^k 
\frac{(k)!}{j! (k-j)!}
\Big( \partial_t^j \Mbc \Big)\Big|_{\partial\Sigma_0} h_{k-1} =0$$
for all $k\geq 1$ where 
$$ h_k\define 
\sum_{j=0}^{k-1} 
\frac{(k-1)!}{j! (k-1-j)!}
 (\partial^j_t \Ham)|_{\partial\Sigma_0} \, h_{k-1-j} + \partial_t^p \big(-\imath \beta \gamma(e_0)f)|_{\partial\Sigma_0}
$$ 
with $h_0=h$. 
\end{remark}

With the above definitions the following corollary follows directly by localizing a solution in any set $U_p$ defined as above and then applying \cite[Theorem 3.1, see also p.~305]{RM}.

\begin{corollary}[Local smooth solution]\label{smooth small}
Let $U_p$ as above. Then a strong solution  for the Cauchy problem~\eqref{CauchyK} is smooth on $U_p$ if and only if $\h\in \Gamma_c(U_p, \mathbb C^N)$ and $\f\in\Gamma(U_p, \mathbb C^N$) satisfy
\begin{equation}\label{compatibility cond}
 \sum_{j=1}^k 
 \frac{k!}{(k-j)!}
( \partial_t^j \oM )|_{\partial\Sigma_0}\, \h_{k-1} =0
\end{equation}
with $\h_i$ as in \eqref{eq_hk}.
\end{corollary}

\begin{remark}\label{rem:compcond}
 If we choose  initial data $\h$ and $\f$ with compact support in the interior of $\Sigma_0$ resp. $\M$, the compatibility condition is automatically satisfied. Actually for $\f$ it is enough to be zero in a neighborhood of $\partial \Sigma_0\subset \M$.
 \end{remark}

Since the compatibility equation~\eqref{compatibility cond} is an 'if and only if'-criterion for smoothness, we can use iteratively already obtained smoothness for $U_p$ with $t(p)<s$  to obtain that the initial data on $U_q$ with $t(p)=s$ fulfills the compatibility criterion as well. That way we will obtain smoothness on the full time strip: 
 
\begin{corollary}\label{smooth_on_T}
 A weak solution of the Cauchy problem~\eqref{CauchyK_2} for $\f\in \Gamma_{cc}(S\M)$ and $\h\in \Gamma_{cc}(S\Sigma_0)$ in an open subset $\mathcal{U}$ of $\mathcal{T}$  is smooth. In particular, there is a smooth solution of the Dirac Cauchy problem~\eqref{CauchyDir} in $\mathcal{U}$. 
 
 The above statement remains true for all $\f\in \Gamma_{c}(S\M)$ and $\h\in \Gamma_{c}(S\Sigma_0)$  fulfilling \eqref{compatibility cond}.
\end{corollary}

\begin{proof}
 First let $p\in \partial \M\cap \Sigma_{\hat{t}}$ for some $\hat{t}\in
[0,T]$ and let $\varrho\colon [0,\hat{t}]\to \partial M$ be a timelike curve
with $\varrho(0)\in \Sigma_0$ and $p=\varrho(\hat{t})$. 
We fix $\epsilon>0$ such that 
we have Fermi coordinates on a 'cube' $U_{\varrho(t)}$ around $\varrho(t)$
as in Section~\ref{diffsol} for all $t\in [0,\hat{t}]$ and such that Corollary~\ref{cor_Up} holds for those cubes. This is always possible since  the image of 
$\varrho$ is compact and everything depends smoothly on the basepoints.

For $U_{\varrho(0)}$ we know that the compatibility condition~
\eqref{compatibility cond} is fulfilled by assumption. Thus, Corollary~
\ref{smooth small} tells us that the weak solution $\Psi$ is smooth in
$U_{\varrho(0)}$ and that for every $a\in [0,\epsilon]$ the function
$\h_a\define \Psi|_{U_{\varrho(0)}\cap \Sigma_{a}}$ as new initial data
$\h$ together with the original $\f$ still fulfill the compatibility
condition. Moreover, $\Psi|_{U_{\varrho(a)}}$ is still a weak solution to
the initial data $(\h_a, \f)$ on $U_{\varrho(a)}$. Thus, we can again use
Corollary~\ref{smooth small} where $\Sigma_a$ now takes the role of $\Sigma_0$. Iterating this procedure, we obtain smoothness on
all $U_{\varrho(t)}$ for $t\in [0,\hat{t}]$, i.e. in particular in $p$. 

For $p\in \M\setminus \partial \M$ we choose a timelike curve $\varrho
\colon [0,\hat{t}]\to \M\setminus \partial \M$  with $\varrho(0)\in
\Sigma_0$ and $p=\varrho(T)$ and proceed as before. It is even easier since we can just use geodesic normal coordinates in the Cauchy surfaces around each $\rho(t)$. The existence of smooth solutions to the Dirac Cauchy problem~\eqref{CauchyDir2} then follows by Lemma~\ref{Equivalence Dir and K}.
\end{proof}

\begin{remark}\label{rem:bd2}
In view of Remark~\ref{rem:bd}, we want to comment on the assumptions on the boundary condition that we have used up to here. The energy inequalities~\eqref{Energy Inequality} and~\eqref{Energy Inequality2} need that $\oM\Psi|_{\bM}=0$ and $\oM^\dagger\Psi|_{\bM}=0$ both imply 
  $$\fiber{\Psi}{\gamma(e_0)\gamma(\n)\Psi}_q =0$$
  for all $q\in\bM$. Moreover, in order to apply \cite{LP} in Lemma~\ref{lem:strong} and \cite{RM} in Corollary~\ref{smooth small}, we additionally use that $\ker\oM|_q$ is nonempty and varies smoothly with $q\in\bM$. \\
  The properties collected above are valid for the MIT bag boundary condition $\Mbc\define \gamma(\n) -\imath$ (and analogously for $\Mbc\define \gamma(\n) +\imath$). Indeed, on account of
\begin{align*}
\fiber{\Psi }{\gamma(e_0) \gamma(\n) \Psi}_q = \fiber{\gamma(e_0) \Psi }{ \gamma(\n) \Psi}_q &=\fiber{- \gamma(\n) \gamma(e_0)  \Psi }{ \Psi}_q \\ &= \fiber{\gamma(e_0) \gamma(\n) \Psi }{ \Psi}_q = \overline{\fiber{\Psi }{\gamma(e_0) \gamma(\n) \Psi}_q},
\end{align*}
where in the third equality we used $g(e_0, \n)=0$, we obtain $  \fiber{\Psi }{\gamma(e_0) \gamma(\n) \Psi}_q \in \RR\,.$
By using MIT boundary conditions, we have
\begin{align*}
\fiber{\Psi }{\gamma(e_0) \gamma( \n )\Psi}_q &= \fiber{\Psi}{ \gamma(e_0) \imath \Psi}_q=\fiber{\gamma(e_0)\Psi}{  \imath \Psi}_q=\\
&=\fiber{- \gamma(e_0) \imath\Psi}{  \Psi}_q=\fiber{-\gamma(e_0) \gamma(\n) \Psi}{  \Psi}_q=-\overline{\fiber{\Psi}{ \gamma(e_0) \gamma(\n)\Psi}_q }\,,
\end{align*}
which implies $\fiber{\Psi }{\gamma(e_0) \gamma(\n) \Psi}_q=0$ for $\oM\Psi=0$.
The rest follows since $\oM=\oM^\dagger$ as we will see in the following: Note that this in particular implies that $\Ham$ as in Remark~\ref{star} restricted to a fixed Cauchy surface $\Sigma_t$ is essentially self-adjoint. First we rewrite the boundary condition $\oM\Psi=0$ as 
$$P_+\Psi:=\frac{1}{2}(\Id+\imath\gamma(\n))\Psi=0$$
 and we set $P_- \Psi:=\frac{1}{2}(\Id-\imath \gamma(\n))\Psi=0$. It is easy to see that
$$P_++P_-=\Id \qquad P_\pm^2=P_\pm \qquad \text{and}\qquad P_+P_-=P_-P_+$$
and in particular that
$\fiber{P_+\Psi }{P_-\Phi}_q=0$ for all $\Psi, \Phi\in S_q\M$. 
This implies that 
$P_+\Psi=0$ is tantamount to set $\Psi=P_- \tilde{\Psi}$ for some $\tilde\Psi$.
Thus,  for $\Psi$ verifying $\oM\Psi=0$, the condition $\fiber{\Phi }{\gamma(e_0) \gamma( \n )\Psi}_q =0$ holds if and only if for any $\tilde{\Psi}$ it is satisfied
\begin{align*}
0&=2\fiber{\Phi }{\gamma(e_0) \gamma( \n )P_- \tilde\Psi}_q=\fiber{\gamma(e_0)\Phi }{( \gamma( \n )+\imath) \tilde\Psi}_q=
\\& =\fiber{(- \gamma( \n )-\imath)\gamma(e_0)\Psi }{ \tilde\Phi}_q=\fiber{\gamma(e_0)( \gamma( \n )-\imath)\Phi }{ \tilde\Psi}_q\,,
\end{align*}
which implies $ \oM\Phi=(\gamma( \n )-\imath)\Phi =0\,.$

Another example is the chirality operator $\Mbc \define (\text{Id}-\gamma(\n) \mathcal{G})$ where $\mathcal{G}$ is the
restriction to $\partial \M$ of an endomorphism-field of $S\M$ which is involutive,
unitary, parallel and anti-commuting with the Clifford multiplication on $\M$, \cite[Section 1.5]{Gin}.
\end{remark}

\subsection{Existence of a weak solution}

Up to now we have seen in Corollary~\ref{smooth_on_T}, Proposition~\ref{prop:finite} and Proposition~\ref{prop:unique}, that if there exists a weak solution $\Psi$ of the Cauchy problem~\eqref{CauchyK}, compare Definition~\ref{def:Weak existence}, then it is smooth, unique and it vanishes outside $\mathcal{V}$ defined by~\eqref{def:V}. It just remains to prove existence.

\begin{thm}[Weak existence]\label{thm:Weak existence}
There exists a unique weak solution $\Psi\in\cH$ to the Cauchy problem~\eqref{CauchyK} with $\f\in \Gamma_{cc}(S\M)$ and $\h\equiv 0$, restricted to $\T$.
\end{thm}

\begin{proof}  Let $\mathcal{U}$ be a compact subset with $\supp \f \Subset \mathcal{U}$ and set  $\mathcal{V}\define J^+(\supp \f)\cap \T$ and  $
\mathcal{V}'\define J^+(\mathcal{U})\cap \T$. Let $\Rv(U)\define \T\cap J^+( J^-(U \cap \Sigma_T) \cap \Sigma_0)$ be defined for any $U\Subset \T$.  Note that $\Rv\define \Rv(\mathcal{V})\Subset \Rv(\mathcal{V}')$.  First we will show that it is enough to find a $\Psi$ such that 
\begin{equation}\label{scalprodRv}
 \scalar{\Phi}{\f}_{\Rv(\mathcal{V}')}=\scalar{\oK^\dagger \Phi}{\Psi}_{\Rv(\mathcal{V}')}
\end{equation}
for all $\Phi\in \Gamma_c(SW)$ where $W\define \text{interior}(\Rv(\mathcal{V}'))\cup (\Rv(\mathcal{V}')\cap \partial \T)$  satisfying $M^\dagger\Phi|_{\partial\M}=0$ and $\Phi|_{\Sigma_T}=0$: Then, by Definition~\ref{def:Weak existence}  $\Psi$ is a weak solution of \eqref{CauchyK} both restricted to $W$ and restricted to $\hat{W}\define \text{interior}(\Rv)\cup (\Rv\cap \partial \T)$.  Hence, by Lemma~\ref{lem:strong} and Corollary~\ref{smooth_on_T} $\Psi$ is a smooth solution on $W$. Using Proposition~\ref{prop:finite} we obtain  that $\Psi|_{\hat{W}\setminus W}=0$. Hence, we can extend $\Psi$ by zero outside $\hat{W}$ to obtain a weak solution on all of $\T$. \medskip 

Hence, it remains to find a weak solution on $\Rv(\mathcal{V}')$: 
 By Lemma~\ref{lemma: energy estimates2} and using arguments similar to Proposition~\ref{prop:unique}, we notice that the  kernel of the operator $\oK_{\lambda}^\dagger$ acting on $$\text{dom}\,  \oK_{\lambda}^\dagger\define \{\Phi\in \Gamma_c(S\T)\ |\ \Phi|_{\Sigma_T}=0, \Mbc^\dagger \Phi|_{\partial \M}=0\}$$ is trivial.  Let now $\ell\colon \oK_{\lambda}^\dagger(\text{dom}\,  \oK_{\lambda}^\dagger)	\to \CC$ be the linear functional defined by
$$\ell(\Theta)=\scalar{\Phi}{\f}_{\Rv(\mathcal{V}')}$$
where $\Phi$ satisfies $\oK_{\lambda}^\dagger \Phi=\Theta$. By the energy inequality~\eqref{Energy Inequality2}, $\ell$ is bounded:
\begin{align*}\ell(\Theta)=&\scalar{\Phi}{\f}_{\Rv(\mathcal{V}')} \leq \|\f\|_{L^2(\Rv(\mathcal{V}'))} \, \|\Phi\|_{L^2(\Rv(\mathcal{V}'))} \\ \leq& c \|\f\|_{L^2(\Rv(\mathcal{V}'))} \|\oK_{\lambda}^\dagger \Phi\|_{L^2(\Rv(\mathcal{V}'))} = c \|\f\|_{L^2(\Rv(\mathcal{V}'))} \|\Theta\|_{L^2(\Rv(\mathcal{V}'))} ,\end{align*}
where in the first inequality we used Cauchy-Schwartz inequality. 
Then $\ell$ can be extended to a continuous functional defined on the $L^2$-completion of $\oK_{\lambda}^\dagger(\text{dom} \, \oK_{\lambda}^\dagger)$ denoted by  $\mathscr{H}\subset \mathcal{H}$. Finally, by Riesz representation theorem, there exists a unique element $\Psi\in \mathcal{H}$ such that 
$$ \ell({\Theta}) =  \scalar{\Theta}{\Psi}_{\Rv(\mathcal{V}')} \,. $$ 
for all $\Theta\in \oK_{\lambda}^\dagger(\text{dom}\, \oK_{\lambda}^\dagger)$. Thus, we obtain
$$ \scalar{\Phi}{\f}_{\Rv(\mathcal{V}')}= \ell({\Theta})=\scalar{\Theta}{\Psi}_{\Rv(\mathcal{V}')} =\scalar{\oK_{\lambda}^\dagger \Phi}{\Psi}_{\Rv(\mathcal{V}')}$$
for all $\Phi\in \text{dom} \oK_{\lambda}^\dagger$. This concludes our proof.
\end{proof}

\section{Global well-posedness of Cauchy problem}
\label{sec:gs}
Up to now we obtained a weak solution in any time strip $t^{-1}([0, T])$  in Theorem~\ref{thm:Weak existence} and showed that it is actually smooth if the initial data are compactly supported in the interior (or more generally fulfill the compatibility condition~\eqref{compatibility cond}, compare Remark~\ref{rem:compcond}). 
We can now easily put everything together to  obtain global well-posedness of the Cauchy problem~\eqref{CauchyDir}--the only comment missing is about negative time:

\begin{remark}\label{rem-negT}When considering~\eqref{CauchyDir} on $t^{-1}([-T,0])$ for some $T>0$ one can map this to the problem in forward time by  a time reversal: If $t\mapsto -t$, $\partial_t\mapsto -\partial_t$ and $e_0\mapsto -e_0$. In particular $-dt$ is then future-directing which makes $\sigma_{\oK_\lambda}(-dt)$ still positive definite. Thus,  going again the way via $\oK_\lambda$ we obtain a smooth unique solution of the original Dirac problem \eqref{CauchyDir} on all $t^{-1}([-T,0])$ for all $T>0$.
\end{remark}

\begin{proof}[Proof of Theorem~\ref{maintheorem}]
Fix $h\in\Gamma_{cc}(S\Sigma_0)$. 
On account of Theorem~\ref{thm:Weak existence}, for any $T\in [0,\infty)$ there exists a weak solution $\Psi_T$ to the Cauchy problem~\eqref{CauchyK} in the time strip $\T_T\define t^{-1}([0,T])$. Combining Corollary~\ref{cor_Up} with Corollary~\ref{smooth_on_T}, we get in particular that $\Psi_T$ is smooth in the time strip $\T_T$. By applying Lemma~\ref{Equivalence Dir and K}, we finally obtain a smooth solution $\psi_T$ to the Dirac Cauchy problem~\eqref{CauchyDir} in the time strip $\T_T$. By uniqueness of the solution, see Proposition~\ref{prop:unique}, we have $\Psi_{T_1}|_{t^{-1}[0,T_1]} =\Psi_{T_2}|_{t^{-1}[0,T_1]}$ for all $T_1<T_2$. Hence, we obtain a smooth solution $\Psi$ of ~\eqref{CauchyDir} on $t^{-1}([0,\infty))$. By Remark~\ref{rem-negT} we get analogously a smooth solution $\tilde{\Psi}$ on $t^{-1}((-\infty, 0])$.  

It is left to show that $\psi$ and $\tilde{\Psi}$ glue smoothly at $t=0$. For that  let $\phi$ be the solution of $\Dir\phi =f$ on $t^{-1}([-1,1])$, $\phi|_{\Sigma_{-1}=t^{-1}(-1)}=\tilde{\Psi}|_{\Sigma_{-1}}$ and $M\phi=0$ on $\partial \M\cap t^{-1}([-1,1])$.  By uniqueness we have $\phi|_{t^{-1}([-1,0])}=\tilde{\Psi}|_{t^{-1}([-1,0])}$ and, hence, $\phi|_{\Sigma_0}=\tilde{\Psi}_{\Sigma_0}=\Psi_{\Sigma_0}$. 
Again by uniqueness $\phi|_{t^{-1}([0,1])}={\Psi}|_{t^{-1}([0,1])}$ and, thus, by smoothness of $\phi$, the solution $\Psi$ and $\tilde{\Psi}$ glue to a smooth solution of  \eqref{CauchyDir} on all of $\M$.

For the continuous dependency on the initial data, see  below.
\end{proof}

We are now in the position to discuss the stability of the Cauchy problem. Since the proof is independent on the presence of the boundary and it does rely mostly on functional analytic techniques, we shall omit it and we refer to \cite[Section 5]{Ba} for further details.
\begin{proposition}
Consider a globally hyperbolic spacetime $\M$ with boundary $\bM$ and denote with $S\M$ the spinor bundle over $\M$. Moreover,  let 
$\Gamma_{0}(S\Sigma_0)\times \Gamma_{0}(S\M)$ denote the space of data satisfying the compatibility condition~\eqref{compatibility cond}. Then the map
$$ \Gamma_{0}(S\Sigma_0) \times \Gamma_0(S\M) \to  \Gamma_{sc}(S\M)$$
 which assign to $  (\h,\f) $ a solution $\Psi$ to the Cauchy problem~\eqref{CauchyK} is continuous.
\end{proposition}

A byproduct of the well-posedness of the Cauchy problem is the existence of Green operators:
 
\begin{proposition}
The Dirac operator is Green hyperbolic, i.e., there exist linear  maps, called \emph{advanced/retarded Green operator},
$\Green^\pm\colon \Gamma_{cc}(S\M)  \to \Gamma_{sc}(S\M)$ satisfying
\begin{enumerate}[(i)]
\item $\Green^\pm \circ \Dir\,  f = \Dir\circ \Green^\pm f=f$ for all $ f \in \Gamma_{cc}(S\M)$;
\item  $\supp(\Green^\pm f ) \subset J^\pm (\supp f )$ for all $f \in \Gamma_{cc} (S\M)$,
\end{enumerate}
where $J^\pm$ denote the causal future (+) and past (-).
\end{proposition}

\begin{proof}
Let $f\in \Gamma_{cc}(S\M)$ and choose $t_0\in \RR$ such that $\supp f \subset J^+( \Sigma_{t_0})$.
By Theorem~\ref{maintheorem}, there exists a unique solution $\psi(f)$ to the Cauchy problem 
$$
\begin{cases}
\Dir \psi =  f\\
\psi|_{\Sigma_0} = 0 \\
\Mbc\Psi|_{\bM} =0.
\end{cases}
$$
For $f\in \Gamma_{cc}(S\M)$ we set  
$\Green^+f\define\psi$ and notice that $\Dir\circ\Green^+f= \Dir \psi = f$. Note that by the finite speed of propagation, cf. Proposition~\ref{prop:finite}, $\Green^+f\in \Gamma_{sc}(S\M)$. Moreover, $\Green^+ \circ \Dir \psi= \Green^+f=\psi$ which finishes the proof of (i). By Proposition~\ref{prop:finite}, we obtain $\supp\Green^+f \subset J^+(f)$ and this conclude the proof of (ii). \\
The existence of the retarded Green operator $\Green^-$ is proven analogously.
\end{proof}


\vspace{0.5cm}

\begin{thebibliography}{99}

\bibitem{ake}
L.~Ak\'e, 
\textit{Some global causal properties of certain classes of spacetimes}, PhD. Thesis, University of M\`alaga, (2018).

\bibitem{AFS} 
	L.~Ak\'e, J.~L.~Flores and M.~S\'anchez,
	\textit{Structure of globally hyperbolic spacetimes
	with timelike boundary.} to appear on Rev. Mat. Iberoamericana.
%
\bibitem{Ba} 
C.~B\"ar, 
\textit{Green-hyperbolic operators on globally hyperbolic spacetimes,} Commun.\ Math.\ Phys.\  {\bf 333}, 1585 (2015).

\bibitem{BB} 
C.~B\"ar and W.~Ballmann, 
\textit{Boundary value problems for elliptic differential operators of first order,} Surveys in differential geometry, Vol. XVII, 1-78, Int. Press, Boston, MA  (2012).

\bibitem{BG} 
C.~B\"ar and N.~Ginoux,
\textit{Classical and quantum fields on Lorentzian manifolds,}   Springer Proc.\ Math.\  {\bf 17}, 359 (2011).

\bibitem{bgp} 
C.~B\"ar, N.~Ginoux and F.~Pf\"affle,
\textit{Wave {E}quations on Lorentzian manifolds and {Q}uantization,}   ESI Lectures in Mathematics and Physics (2007).

\bibitem{Beem}
J.K.~Beem, P.E.~Ehrlich and K.L.~Easley,
\textit{Global {L}orentzian geometry}, Monographs and Textbooks in Pure and 
Applied Mathematics (202), 2nd ed., Marcel Dekker, Inc., New York, 1996.
		

\bibitem{BD}
M.~Benini and C.~Dappiaggi, 
\textit{Models of free quantum field theories on curved backgrounds,}
in: R. Brunetti, C. Dappiaggi, K. Fredenhagen and J. Yngvason (eds.), Advances
in Algebraic Quantum Field Theory, 75-124, Springer-Verlag, Heidelberg (2015).

 \bibitem{simo5}
M.~Benini, C.~Dappiaggi and S.~Murro,
\textit{Radiative observables for linearized gravity on asymptotically flat spacetimes and their boundary induced states.} J.\ Math.\ Phys.\ \textbf{55}, 082301 (2014). 

\bibitem{aqft} 
  M.~Benini, C.~Dappiaggi and A.~Schenkel
  \textit{Algebraic quantum field theory on spacetimes with timelike boundary,}  
  Ann. Henri Poincar\'e 19, no. 8, 2401-2433 (2018)


  \bibitem{BS}
   A.N.~Bernal and M.~S\'anchez,
   \textit{Smoothness of time functions and the
metric splitting of globally hyperbolic spacetimes,}
   Comm. Math. Phys. \textbf{257}, 43-50 (2005).
   
    \bibitem{aqft2} 
R.~Brunetti, C. Dappiaggi, K. Fredenhagen and J. Yngvason, \textit{Advances in algebraic quantum field theory.} Springer (2015)
   
\bibitem{Dappia3}
 F.~Bussola, C.~Dappiaggi, H.~R.~C.~Ferreira and I.~Khavkine, 
\textit{Ground state for a massive scalar field in the BTZ spacetime with Robin boundary conditions,} Phys.\ Rev.\ D\ \textbf{96}, 105016 (2017).

\bibitem{MIT2} 
  A.~Chodos, R.~L.~Jaffe, K.~Johnson and C.~B.~Thorn,
  \textit{Baryon Structure in the Bag Theory,}
  Phys.\ Rev.\ D {\bf 10}, 2599 (1974).
  
  \bibitem{MIT1} 
  A.~Chodos, R.~L.~Jaffe, K.~Johnson, C.~B.~Thorn and V.~F.~Weisskopf,
  \textit{New extended model of hadrons,}
  Phys.\ Rev.\ D {\bf 9}, 3471 (1974).
  
    \bibitem{Choquet} 
Y.~Choquet-Bruhat,
\textit{Hyperbolic differential equations on a manifold,} Chapter in Battelle Rencontres (DeWitt and Wheeler, eds.), Benjamin, New York, (1968).
 
\bibitem{CGS}  
P.~T.~Chru\'sciel, G.~J.~Galloway and D.~Solis, \textit{Topological censorship for Kaluza–Klein
space-times,} Annales Henri Poincar\'e 10 (2009) 893.


 \bibitem{Dappia} 
C.~Dappiaggi, N.~Drago and H.~Ferreira
  \textit{Fundamental solutions for the wave operator on static Lorentzian manifolds with timelike boundary,}  Lett.\ Math.\ Phys.\ \textbf{109}, 2157-2186 (2019).
  
 \bibitem{Dappia1} 
 C.~Dappiaggi and H.~R.~C.~Ferreira, 
\textit{ Hadamard states for a scalar field in anti-de Sitter spacetime with arbitrary boundary conditions,}  Phys.\ Rev.\ D\ \textbf{94}, 125016 (2016). 

 \bibitem{Dappia2} 
 C.~Dappiaggi and H.~R.~C.~Ferreira, 
\textit{On the algebraic quantization of a massive scalar
field in anti-de-Sitter spacetime,} 
Rev.\ Math.\ Phys.\ \textbf{30} 1850004 (2018)  .

    \bibitem{simo4}
C. Dappiaggi, F.~Finster, S.~Murro and E.~Radici,
{\it The Fermionic Signature Operator in De Sitter Spacetime,}
 J.\ Math.\ Anal.\ Appl. \textbf{485}, 123808 (2020).
  
  \bibitem{DHP}
  C.~Dappiaggi, T.~P.~Hack and N.~Pinamonti
  \textit{The extended algebra of observables for Dirac fields and the trace anomaly of their stress-energy tensor,} 
  Rev.\ Math.\ Phys.\ \textbf{21}, 1241-1312 (2009).
  
  
\bibitem{simo6}
C.~Dappiaggi, S.~Murro and A.~Schenkel,
\textit{Non-existence of natural states for Abelian Chern–Simons theory,}
J.\ Geom.\ Phys.\ \textbf{116}, 119-123 (2017).


\bibitem{DappiaCas} 
  C.~Dappiaggi, G.~Nosari and N.~Pinamonti,
  \textit{The Casimir effect from the point of view of algebraic quantum field theory,}
  Math.\ Phys.\ Anal.\ Geom.\  {\bf 19}, 12 (2016).


\bibitem{simo3}
 N.~Drago and S.~Murro,
  \textit{A new class of Fermionic Projectors: M{\o}ller operators and mass oscillation properties,}
  Lett.\ Math.\ Phys.\ \textbf{107}, 2433-2451 (2017).
  
   
    \bibitem{dimock} 
  J.~Dimock
  \textit{Dirac quantum fields on a manifold,}
Trans.\ Am.\ Math.\ Soc.\  {\bf 269}, 133 (1982).
  
   \bibitem{Evans}
L.~C.~Evans, 
\textit{Partial differential equations,} 
American Mathematical Society (2010).

\bibitem{FK} 
  K.~Fredenhagen and K.~Rejzner,
  \textit{Quantum field theory on curved spacetimes: Axiomatic framework and examples,}
  J.\ Math.\ Phys.\  {\bf 57}, 031101 (2016).

\bibitem{FKSY} 
  F.~Finster, N.~Kamran, J.~Smoller and S.~T.~Yau,
  \textit{The Long time dynamics of Dirac particles in the Kerr-Newman black hole geometry,}
  Adv.\ Theor.\ Math.\ Phys.\  {\bf 7}, no. 1, 25 (2003).

\bibitem{simo1}
F.~Finster, S.~Murro and C.~R\"oken,
{\it The Fermionic Projector in a Time-Dependent External Potential: Mass Oscillation Property and Hadamard States,}
J.\ Math.\ Phys.\  {\bf 57}, 072303 (2016).

  \bibitem{simo2}
F.~Finster, S.~Murro and C.~R\"oken,
{\it The Fermionic Signature Operator and Quantum States in Rindler Space-Time,}
  J.\ Math.\ Anal.\ Appl.\  {\bf 454}, 385 (2017).
  

\bibitem{FR1} 
F.~Finster and C.~R\"oken, \textit{Self-adjointness of the Dirac Hamiltonian for a class of non-uniformly elliptic boundary value problems,} Annals of Mathematical Sciences and Applications \textbf{1} , 301 (2016).

\bibitem{FR2} 
  F.~Finster and C.~Röken,
  \textit{An integral spectral representation for the massive Dirac propagator in Kerr geometry in Eddington-Finkelstein-type coordinates,} Adv. Theor. Math. Phys. 22, no. 1, 47-92 (2018).

\bibitem{Fr1} 
K.~O.~Friedrichs, \textit{Symmetric hyperbolic linear differential equations,} Comm.\ Pure\ Appl.\ Math.\  {\bf 7}, 345 (1954).

\bibitem{Fr2} 
K.~O.~Friedrichs, \textit{Symmetric positive linear differential equations,} Comm.\ Pure\ Appl.\ Math.\  {\bf 11}, 333 (1958).
 
  \bibitem{gerard} 
C. G\'erard, \textit{Microlocal Analysis of Quantum Fields on Curved Spacetimes.} ESI Lectures in Mathematics and Physics (2019). 

\bibitem{Ge}
 R.~Geroch,
\textit{Domain of dependence}, J. Math. Phys. \textbf{11},
437-449 (1970).

\bibitem{Gin}
N.~Ginoux,
\textit{The Dirac spectrum}, Lecture Notes in Math., Springer (2009).


 \bibitem{GS} 
 N.~Gro{\ss}e and C. Schneider, \textit{Sobolev spaces on Riemannian manifolds with bounded geometry: General coordinates and traces,} Math.\ Nachr.\  {\bf 286}, 1586-1613 (2013).
 
\bibitem{MIT3} 
  R.~W.~Haymaker and T.~Goldman,
  \textit{Bag boundary conditions for confinement in the $q\overline{q}$ relative coordinate,}
  Phys.\ Rev.\ D {\bf 24}, 743 (1981).

\bibitem{IR} 
  G.~Idelon-Riton,
  \textit{Scattering theory for the Dirac equation on the Schwarzschild-anti-de Sitter space-time,} Adv. Theor. Math. Phys. 22, no.4, 1007-1069 (2018).

\bibitem{LP} 
P.~D.~Lax and R.~S.~Phillips,
 \textit{Local boundary conditions for dissipative symmetric linear differential operators,} Comm.\ Pure\ Appl.\ Math.\  {\bf 13}, 427 (1960).

  
  
 \bibitem{Leray}
J.~Leray, 
\textit{Hyperbolic differential equations,} 
Institute for advanced study, (1955).

\bibitem{Lee}
J.~M.~Lee,
\textit{Introduction to smooth manifolds,}
Springer Science+Business Media New York 2003, (2013).

 \bibitem{Nicolas}
J.~P.~Nicolas, 
\textit{Dirac fields on asymptotically flat space-times,} 
Dissertationes\ Math.\ {\bf 408}, 85 (2002).
  
 \bibitem{PV}
  S.~Pigola and G.~Veronelli,
  \textit{The smooth Riemannian extension problem: completeness,}
to appear on Ann.\ Scuola Norm.\ Sup.\ Pisa Cl.\ Sci.\ .
  
    \bibitem{RM} 
 J.~B.~Rauch and F.~J.~Massey,
 \textit{Differentiability of solutions to hyperbolic initial-boundary value problems,}
Trans.\ Am.\ Math.\ Soc.\  \textbf{189}, 303 (1974).

\bibitem{Casimir} 
A.~Seyedzahedi, R.~Saghian and S.~S.~Gousheh,
  \textit{Fermionic Casimir energy in a three-dimensional box,}
  Phys.\ Rev.\ A {\bf 82}, 032517 (2010).

  \bibitem{Sol}
D.~Solis, 
\textit{Global properties of asymptotically de Sitter andAnti de Sitter spacetimes}, PhD. Thesis, 2006, arXiv: 1803.01171.

  \bibitem{Zahn} 
J.~ Zahn,
\textit{Generalized Wentzell Boundary Conditions and Quantum Field Theory,}
Annales Henri Poincar\'e, \textbf{19},  163–187 (2018).
\end{thebibliography}
\end{document}